\documentclass[11pt]{amsart}

\usepackage{pb-diagram}

\usepackage{amsfonts,amscd}
\usepackage[mathscr]{eucal}
\usepackage{color}
\usepackage{amsmath}
\usepackage{amssymb}

\newcommand{\tr}{\textnormal{tr}}

\newcommand{\al}{\alpha}
\newcommand{\be}{\beta}

\newcommand{\de}{\delta}
\newcommand{\ep}{\epsilon}

\newcommand{\osc}{\mathrm{osc}}
\newcommand{\PSH}{\mathrm{PSH}}
\newcommand{\mt}{\mathfrak{mt}}

\newcommand{\dbar}{\overline{\partial}}
\newcommand{\Dd}{\partial}

\newcommand{\ddt}[1]{\frac{\partial #1}{\partial t}}

\newcommand{\Ricc}{\mathrm{Ric}}

\newcommand{\ZZ}{\mathbb{Z}}

\newcommand{\PP}{\mathbb{P}}

\newcommand{\FF}{\mathcal{F}}

\newcommand{\ddbar}{\sqrt{-1}\partial\dbar}

\newtheorem{example}{Example}[section]
\newtheorem{theorem}{Theorem}[section]
\newtheorem{proposition}{Proposition}[section]
\newtheorem{lemma}{Lemma}[section]
\newtheorem{definition}{Definition}[section]
\newtheorem{corollary}{Corollary}[section]
\newtheorem{remark}{Remark}[section]
\newtheorem{conjecture}{Conjecture}[section]

\def\sO{{\mathscr O}}

\def\sO{\mathscr{O}}

\begin{document}

\title{The greatest Ricci lower bound, conical Einstein metrics and the Chern number inequality}

\author{Jian Song}

\address{Department of Mathematics, Rutgers University, Piscataway, NJ 08854}

\email{jiansong@math.rutgers.edu}

\thanks{The first named author is supported by National Science Foundation grants DMS-0847524 and a Sloan Foundation Fellowship.}

\author{Xiaowei Wang}

\address{Department of Mathematics, Rutgers University, Newark, NJ 07102}

\email{xiaowwan@rutgers.edu}

\begin{abstract}  We partially confirm a conjecture of Donaldson relating the greatest Ricci lower bound $R(X)$ to the existence of conical K\"ahler-Einstein metrics on a Fano manifold $X$. In particular, if $D\in |-K_X|$ is a smooth simple divisor and the Mabuchi $K$-energy is bounded below,  then  there exists a unique conical K\"ahler-Einstein metric satisfying $\Ricc(g) = \beta g + (1-\beta) [D]$ for any $\beta \in (0,1)$. We also construct unique smooth conical toric K\"ahler-Einstein metrics with $\beta=R(X)$ and a unique effective $\mathbb{Q}$-divisor $D\in [-K_X]$ for all toric Fano manifolds. Finally we prove a Miyaoka-Yau type inequality for Fano manifolds with $R(X)=1$. 

\end{abstract}

\maketitle

{\footnotesize  \tableofcontents}

\section{Introduction}

The existence of K\"ahler-Einstein metrics has been a central problem in K\"ahler geometry since Yau's celebrated solution \cite{Y1} to the Calabi conjecture. Constant scalar curvature metrics with conical singularities  have been extensively studied in \cite{Mc, Tr, LT} for Riemann surfaces. In general, we can consider a pair $(X, D)$ for an $n$-dimensional compact K\"ahler manifold and a smooth complex hypersurface $D$ of $X$. A conical K\"ahler metric $g$ on $X$ with cone angle $2\pi \beta$ along $D$ is locally equivalent to the following model edge metric
$$g= \sum_{j=1}^{n-1} dz_j \otimes d\bar z_j + |z_n|^{-2(1-\beta)} d z_n \otimes d\bar z_n,$$
if $D$ is defined by $z_n=0$.
Applications of conical K\"ahler metrics are proposed \cite{Ts, T94} to obtain various Chern number inequalities. Recently, Donaldson developed the linear theory to study the existence of canonical conical K\"ahler metrics  in \cite{D4}.  And  Brendle in \cite{Br} solved Yau's Monge-Amp\`ere equations for conical K\"ahler metrics with cone angle $2\pi \beta$ for $\beta\in (0,1/2)$ along a smooth divisor $D$. The general case is settled by Jeffres, Mazzeo and Rubinstein \cite{JMR} for all $\beta\in (0,1)$. As an immediate consequence, there always exists conical K\"ahler-Einstein metrics with negative or zero constant scalar curvature with cone angle $2\pi \beta$ along a smooth divisor $D$ for $\beta\in (0,1)$.
When $X$ is a Fano manifold, Donaldson'proposes to study the conical K\"{a}hler-Einstein equation
\begin{equation}\label{KE-D}
\Ricc(\omega) = \be \omega + (1-\be) [D],
\end{equation}
where $D$ is smooth simple divisor in the anticanonical class $[-K_X]$ and $\beta \in (0, 1)$.
One of the motivations is that one can study the existence problem for smooth
K\"{a}hler-Einstein metrics on $X$ by deforming the cone angle. Such an approach can be regarded as a variant of the standard continuity method.

In particular, since Tian and Yau have already established the existence of
a complete Ricci-flat K\"{a}hler metric on the non-compact manifold $X\setminus D$  in \cite{TY}, one would expect that \eqref{KE-D} is solvable for $\be$ sufficient small. This is comfirmed successfully by  Berman  in \cite{Be}. Now the question is how large $\be$ can be.  The largest  $\be$  is closely related to the following holomorphic invariant known as the greatest Ricci lower bound first introduced by Tian in \cite{T92}.

\begin{definition}\label{Rg}
Let $X$ be a Fano manifold.  The greatest Ricci lower bound $R(X)$ is defined by
\begin{equation}
 R(X)=\sup\{\beta\mid \Ricc(\omega)\geq \beta \omega, \text{ for some smooth K\"{a}hler metric }\omega\in c_1(X)\}.%
 \end{equation}

\end{definition}
It is proved by Szekelyhidi in \cite{Sze} that $[0, R(X))$ is the maximal interval for the continuity method to solve the K\"ahler-Einstein equation on a Fano manifold $X$. In particular, it is independent of the choice for the initial K\"ahler metric when applying the continuity method. The invariant $R(X)$ is explicitly calculated for $\PP^2$ blown up at one point by Szekelyhidi \cite{Sze}, and for all toric Fano manifolds by Li \cite{LiC}.  It is well-known that  Mabuchi $K$-energy being bounded from below implies $R(X)=1$ and it is proved in \cite{MS} that  $R(X)=1$ implies  $X$ being $K$-semistable.  The following conjecture is proposed by Donaldson in \cite{D4} to relate $R(X)$ to the existence of conical K\"ahler-Einstein metrics.
\begin{conjecture}\label{D-conj}There does not exist a conical K\"ahler-Einstein metric solving (\ref{KE-D})  if $\beta\in (R(X),1]$, while there exists one if $\beta\in (0, R(X))$.

\end{conjecture}

This conjecture can be considered as a different geometric interpretation of  the invariant $R(X)$.  Another importance of the conjecture lies in the fact that it supplies a new approach to the Yau's  conjecture \cite{Y3} on the equivalence of existence of K\"{a}hler-Einstein metric for Fano manifolds and certain algebro-geometric stability condition, which is refined and extended by Tian \cite{T97} and Donaldson \cite{D}.  The algebro-geometric aspect of Conjecture \ref{D-conj} has been studied by Li \cite{LiC3}, Sun \cite{Sun}, Odaka-Sun \cite{OS} and  Berman \cite{Be-l}. In particular,  the notion of Log $K$-stability is introduced in  \cite{LiC3} and \cite{Sun}  as the algebro-geometric obstruction to solve equation \eqref{KE-D}. In particular, $R(X)$ can be applied to test the Log $K$-stability of $X$ when it is toric Fano.  In \cite{Be-l}, Berman proves that Log $K$-stability is a necessary condition to the solution of \eqref{KE-D}. This naturally leads to  the Log version of Yau-Tian-Donaldson conjecture, that is, to establish the equivalence of the solvability of \eqref{KE-D} and the Log $K$-stability of $(X,D)$ (cf.\cite{LiC3} and \cite{OS}).  An interesting observation made by Sun in \cite{Sun} is that $K$-stability implies Log $K$-stability.

Now we describe the main results of the present work. The first one is to partially confirm Conjecture \ref{D-conj}. We consider a more general class of  conical K\"ahler-Einstein metrics with  smooth simple divisors  in any pluricanonical systems, and remove the assumption in \cite{D4} on $D$ by showing there exists no holomorphic vector field tangential to $D$ (c.f. Theorem \ref{nohol}).

\begin{theorem}\label{main1}  Let $X$ be a Fano manifold and $R(X)$ be the greatest  lower bound of Ricci curvature of $X$.

\begin{enumerate}

\item  For any $\beta \in [R(X), 1]$ and any smooth simple divisor $D\in |-mK_X|$ for some $m\in \mathbb{Z}^+$, there does not exist a smooth conical K\"ahler-Einstein metric $\omega$ satisfying
\begin{equation}\label{KE}
\Ricc(\omega) = \beta \omega + \frac{ 1-\beta }{m} [  D ],
\end{equation}
if $R(X)<1$.

\medskip

\item

For any $\beta \in (0, R(X))$, there exist  a smooth simple divisor $D\in |-mK_X|$ for some $m \in \mathbb{Z}^+$ and a smooth conical K\"ahler-Einstein metric $\omega$ satisfying \eqref{KE}.

\end{enumerate}

\end{theorem}

The second part of the theorem is not completely satisfactory in the sense that one would like to have an $m$ that is  independent  of $\beta\in (0, R(X))$. In the case when the Mabuchi $K$-energy is bounded below, or more generally $R(X)=1$,  we show that $D$ and $\beta$ do not rely on $\beta$.

\begin{theorem}\label{main2} Let $X$ be a Fano manifold.  If the Mabuchi $K$-energy $\mathcal{M}$ is bounded below and if $D\in |-K_X|$ is a smooth simple divisor,  then for any $\beta\in (0, 1)$ there exists a smooth conical K\"ahler-Einstein metric satisfying the conical K\"ahler-Einstein equation
$$\Ricc(\omega) = \beta \omega + (1-\beta) [D]. $$

In general,  if the paired Mabuchi $K$-energy $\mathcal{M}_{D, R(X)}$ is bounded below  for  a smooth simple divisor $D \in |-m K_X |$ for some $m\in \mathbb{Z}^+$,   then for any $\beta\in (0, R(X))$ there exists a smooth conical K\"ahler-Einstein metric satisfying equation (\ref{KE}).

\end{theorem}

We would like to remark that most results of Theorem \ref{main1} and \ref{main2} are independently obtained by Li and Sun \cite{LS}.  Theorem \ref{main2} might have many applications. In particular, if the Mabuchi $K$-energy is bounded below, there exists a sequence of conical K\"ahler-Einstein metrics $\Ricc(g_\epsilon ) = (1-\epsilon)  g + \epsilon  [D]$ as $ \epsilon \rightarrow 0 $. $(X, g_\epsilon)$ might converge in Gromov-Hausdorff topology to a $\mathbb{Q}$-Fano variety $X_\infty$ coupled with a canonical K\"ahler-Einstein metric.   Theorem \ref{main2} also holds if $R(X)=1$ and $D\in |-mK_X|$ for some $m\geq 2$ as in the following proposition.  By Bertini's theorem, there always exists a smooth simple divisor $D\in |-mK_X|$ for $m$ sufficiently large.

\begin{proposition} \label{rmain1}Let $X$ be a Fano manifold and $D\in |-mK_X|$ be a smooth simple divisor for some $m\geq 2$.   Then for any $\beta\in \left(0, \frac{(m-1)\cdot R(X)}{m-R(X)}\right)$, there exists a smooth conical K\"ahler-Einstein metric $\omega$ satisfying \eqref{KE} for $D$. In particular, when $R(X)=1$, equation \eqref{KE} with $D$ is solvable for  any $\be\in (0,1)$.

\end{proposition}

The invariant $R(X)$ can also be identified as the optimal constant for the nonlinear Moser-Trudinger inequality.
Let $X$ be a Fano manifold and $\omega\in c_1(X)$ be a smooth K\"ahler metric on $X$. One can define the following  $F$-functional by
\begin{equation}
F_{\omega, \beta} = J_\omega(\varphi) - \frac{1}{V} \int_X \varphi \omega^n - \frac{1}{\beta} \log \frac{1}{V} \int_X e^{-\beta \varphi} \omega^n,
\end{equation}
where $$J_\omega(\varphi)=\frac{\sqrt{-1}}{V}\sum_{i=0}^{n-1}\frac{i+1}{n+1}\int_X\Dd\varphi\wedge\dbar\varphi\wedge\omega^i\wedge\omega^{n-1-i}_\varphi,
$$
is the Aubin-Yau functional, $\omega_\varphi=\omega+\sqrt{-1}\Dd\dbar\varphi>0$ and $V= \int_X \omega^n$.
As a corollary of Theorem \ref{main1}, we can establish a connection between $R(X)$ and the Moser-Trudinger inequality.

\begin{enumerate}

\item If $\beta \in (0, R(X))$,  $F_{\omega, \beta} $ is bounded below and $J$-proper (see Definition \ref{strong}) on $PSH(X, \omega) \cap L^\infty(X)$, or equivalently, there exist $\epsilon, C_\epsilon>0$ such that the following Moser-Truding inequality holds%
$$\int_X e^{-\beta \varphi} \omega^n \leq C_\epsilon e^{ (\beta-\epsilon) J_\omega (\varphi) - \frac{\beta}{V} \int_X \varphi \omega^n} $$
for $\varphi\in PSH(X,\omega)\cap L^\infty(X)$.

\item If $\beta\in (R(X), 1)$, then $$\inf_{PSH(X, \omega)\cap L^\infty(X)} F_{\omega, \beta}(\cdot) = - \infty. $$

\end{enumerate}

It is first proved in \cite{T97, TZ} for the properness of $F$-functional on Fano K\"ahler-Einstein manifolds without holomorphic vector fields. The $J$-properness for $F$ is conjectured in this case in \cite{T97} and is later proved in \cite{PSSW}. The presence of the smooth divisor $D$ for $\beta\in (0,1]$ blocks holomorphic vector fields as shown in Theorem \ref{nohol}.  It is interesting to ask if $F_{\omega, \beta}$ is always bounded below on $PSH(X, \omega) \cap L^\infty(X)$ if $\beta = R(X)$.  In the case of toric Fano manifolds, $F_{\omega, \beta}$ is indeed bounded from below if $\beta=R(X)$  as a corollary of the following theorem (Corollary \ref{torlowbd}).
A more interesting problem will be to understand  the limiting behavior of the conical K\"ahler-Einstein metrics as $\beta\rightarrow R(X)$ as the holomorphic vector field will appear in the limiting space. The following theorem serves an example for the above speculation.

\begin{theorem}\label{main3} Let $X$ be a  toric Fano  manifold. Then there exist a unique effective toric $\mathbb{Q}$-divisor $D\in |-K_X|$ and a unique smooth toric conical K\"ahler metric $\omega$ satisfying
\begin{equation}
\Ricc(\omega) = R(X) \omega + (1-R(X)) [D].
\end{equation}
Furthermore, $R(X)$ is the largest number for $\beta\in (0, 1]$ such that
\begin{equation}\label{toreqnb}
\Ricc(\omega) = \beta \omega + (1-\beta) [D_\beta]
\end{equation}
admits a smooth conical toric solution $\omega_\beta$  for   an effective toric $\mathbb{R}$-divisor $D_\beta$  in $[-K_X]$.

\end{theorem}

We remark that the divisor $D$ can not be smooth, instead it is a union of effective smooth toric $\mathbb{Q}$-divisors with simple normal crossings. Theorem \ref{main3} is closely related to the results of \cite{LiC2} with a different approach for the limiting behavior of the continuity method. The proof of Theorem \ref{main3} relies on the toric setting introduced in \cite{D1, D2} and the estimates in \cite{WZ}. For $\beta > R(X)$, there still exists a smooth conical solution for equation (\ref{toreqnb}), however, $D_\beta$ won't be effective and so the Ricci current of the conical metric cannot be positive. In Theorem \ref{main3}, $R(X)$ will be explicitly calculated as in \cite{LiC} and $D$ is determined by Lemma \ref{divisord} and Lemma \ref{torickefr} . For example, when $X$ is $\PP^2$ blown up at one point, it admits a  toric $\PP^1$ bundle $\pi: X \rightarrow \PP^1$, then $R(X)=6/7$ and $D = 2D_\infty + (H_1 + H_2)/2$, where
$D_\infty$ is the section at the infinity, $H_1$ and $H_2$ are the two $\PP^1$ fibres invariant under the torus action. This seems to suggest that Donaldson's conjecture might only hold for  smooth simple divisors lying in the pluri-anti-canonical system. In fact, it is shown in \cite{LS} Theorem \ref{main3} can be applied to prove Conjecture \ref{D-conj} in the toric case when one is allowed to replace $|-K_X|$ by  the linear system of a suitable power  of $-K_X$.

Finally, we will give some applications of Theorem \ref{main2}. In general, the conical K\"ahler metrics do not have bounded curvature tensors, as they might blow up near the conical divisor, particularly when the cone angle is greater than $\pi$. However,  we can show that if  the Ricci curvature of a conical K\"ahler  metric is bounded, its curvature tensors are always bounded in $L^2$-norm.

\begin{proposition} \label{chpro} Let $X$ be a K\"ahler manifold and $D$ be a smooth simple divisor on $X$. Let $g$ be a smooth conical K\"ahler metric on $X$ with cone angle $2\pi \beta$ along $D$ with $\beta\in (0,1)$. If the Ricci curvature of $g$ is bounded, the $L^2$-norm of the curvature tensors of $g$ is also bounded.

\end{proposition}

When we say the Ricci curvature of a conical K\"ahler metric $g$ is bounded, it means that the Ricci curvature of $g$ is uniformly bounded on $X\setminus D$ by ignoring the mass along $D$ in the Ricci current. Proposition \ref{chpro} enables us to define Chern characters, and in particular the Chern numbers for those conical K\"ahler metrics and derive corresponding Gauss-Bonnet and signature formulas for K\"ahler surfaces with conical singularities along a smooth holomorphic curve $\Sigma$. This is related to recent results of Atiyah and Lebrun \cite{AL} for smooth Riemannian $4$-folds.  In fact, the bound of $L^2$-norm of the curvature tensors only depend on the scalar curvature bound and topological invariants such as intersection numbers among $D$, the first, second Chern classes. When the greatest Ricci lower bound $R(X)$ is $1$,  we have the Miyaoka-Yau type inequality holds for $X$ as in the K\"ahler-Einstein case \cite{Mi, Y2}. In such a case, $X$ is $K$-semistable \cite{MS} and the result below reflects the general perspective connecting $K$-semistability and the Chern number inequality.

\begin{theorem}\label{main4} Let $X$ be a Fano manifold.  If $R(X)=1$, then the Miyaoka-Yau type inequality holds
\begin{equation}
c_2(X) \cdot c_1(X)^{n-2} \geq \frac{n} {2(n+1)}  c_1(X)^n.
\end{equation}
In general, if $D\in |-K_X|$ is a smooth simple divisor and if the paired Mabuchi $K$-energy $\mathcal{M}_{D, \beta}$ is bounded below, then
\begin{equation}
 c_2(X) \cdot c_1(X)^{n-2} \geq \frac{n\beta^2} {2(n+1)}  c_1(X)^n.
\end{equation}

\end{theorem}
Parallel argument can be applied to give a complete proof of the Chern number inequality for smooth minimal models of general type using conical K\"ahler-Einstein metrics. This approach is first proposed by Tsuji \cite{Ts},  while the analytic estimates seem missing. We remark that the first complete proof for smooth minimal models of general type is due to Zhang \cite{ZhY} by using the K\"ahler-Ricci flow.

\bigskip

\section{$R(X)$ and conical K\"ahler-Einstein metrics}

\subsection{Paired energy functionals}

In this section, we recall the paired energy functionals originally introduced in \cite{Be}.

\begin{definition}\label{fdef} Let $X$ be a Fano manifold and $\omega\in c_1(X)$ be  a K\"ahler current with bounded local potential and $\Omega_\theta $ be a   nonnegative real $(n,n)$-current on $X$ whose curvature $$\theta = -\ddbar \log \Omega_\theta  \in c_1(X)$$ is a nonnegative $(1,1)$-current. Let $\Omega_\omega$ be a nonnegative real $(n,n)$-current on $X$ defined by $$\ddbar \log \Omega_\omega = \omega.$$
Suppose that for some $\beta\in (0, 1]$,  $$\int_X (\Omega_\omega)^\beta (\Omega_\theta)^{1-\beta} = V = c_1(X)^n.$$
We define the paired $F$-functional by 
\begin{equation}\label{F}
 F_{\omega,\theta, \beta}(\varphi)= J_\omega(\varphi)- \frac{1}{V} \int_X \varphi \omega^n
 -\frac{1}{\beta}\log\frac{1}{V}\int_X \left( e^{-\varphi } \Omega_\omega \right)^\beta \left( \Omega_\theta\right)^{1-\beta} ,
\end{equation}
for $\varphi \in PSH(X, \omega) \cap L^\infty(X)$, 
where $$J_\omega(\varphi)=\frac{\sqrt{-1}}{V}\sum_{i=0}^{n-1}\frac{i+1}{n+1}\int_X\Dd\varphi\wedge\dbar\varphi\wedge\omega^i\wedge\omega^{n-1-i}_\varphi,
$$
is the Aubin-Yau $J$-functional  and $\omega_\varphi=\omega+\sqrt{-1}\Dd\dbar\varphi\geq 0$.
\end{definition}

The Euler-Lagrangian equation for (\ref{F}) is given by
\begin{equation}\label{criteqn1}
(\omega+\ddbar \varphi)^n = (e^{-\varphi} \Omega_\omega)^\beta (\Omega_\theta)^{1-\beta},
\end{equation}
and the corresponding curvature equation is
\begin{equation}\label{criteqn2}
\Ricc(\omega_\varphi ) = \beta \omega_\varphi + (1-\beta) \theta.
\end{equation}

When $\beta =1$,  $F_{\omega, \theta, \beta} (\varphi) = F_\omega$ is the original Ding's functional \cite{Di}. The paired $F$-functional also  satisfies the cocycle condition by slightly modifying the proof for the original $F$-functional. 

\begin{lemma} \label{cocy}
  $F_{\omega, \theta, \beta} $ satisfies the following cocycle condition
\begin{equation}
F_{\omega, \theta, \beta} (\varphi) - F_{\omega_\psi, \theta, \beta} (\varphi - \psi) = F_{\omega, \theta, \beta}(\psi)
\end{equation}
for any $\varphi, \psi \in PSH(X, \omega) \cap L^\infty(X)$.

\end{lemma}

If we write $F(\omega, \omega_\varphi) = F_{\omega, \theta,\beta}(\varphi)$, then the cocycle condition is equivalent to the following 
$$F(\omega, \omega_\varphi) + F(\omega_\varphi, \omega_\psi) + F(\omega_\psi, \omega) =0.$$

\begin{definition}\label{strong}

We say a functional $G( \cdot )$ is $J$-proper on $PSH(X, \omega)\cap L^\infty(X)$  if  there exist $\de,\ C_\de>0$ such that
$$G(\varphi)\geq \de J_{\omega}(\varphi)-C_\de$$
for all $ \varphi\in \PSH(X,\omega)\cap L^\infty(X)$.

\end{definition}

Let $X$ be a Fano manifold and $D$ be a smooth  simple divisor in $|-mK_X|$.  Let $s$ be a defining section of $[D]$.   Since $s\in H^0(X, -mK_X)$, $$\Omega_D= |s|^{- \frac{2}{m}}=(s\otimes \overline{s})^{ - \frac{1}{m}} $$ can be considered as a smooth nonnegative real $(n,n)$-form with poles along $D$ of order $m^{-1}$. Obviously, $\Ricc(\Omega_D) = -\ddbar \log \Omega_D = m^{-1} [D]$.
We then define the following notations for conveniences.

\begin{definition}\label{modf} Let $D\in |-mK_X|$ be a smooth simple divisor for some $m \in \mathbb{Z}^+$. We define
\begin{eqnarray}
\mathcal{F}_{\omega, \beta}(\varphi) &=& F_{\omega, m^{-1}[D], \beta} (\varphi)\\
&=& J_{\omega}(\varphi) - \frac{1}{V} \int_X \varphi \omega^n - \frac{1}{\beta} \log \frac{1}{V} \int_X e^{-\beta \varphi} (\Omega_\omega)^{\beta} (\Omega_D)^{1-\beta} \nonumber
\end{eqnarray}
and 
\begin{equation}
F_{\omega, \beta} (\varphi) = J_{\omega}(\varphi) - \frac{1}{V} \int_X \varphi \omega^n - \frac{1}{\beta} \log \frac{1}{V} \int_X e^{-\beta \varphi} \omega^n.
\end{equation}

\end{definition}

In order to relate the  Moser-Trudinger inequality to $R(X)$, we introduce

\begin{definition}  Let $X$ be a Fano manifold and $\omega\in c_1(X)$ be a smooth K\"ahler metric.   We define the optimal {\em Moser-Trudinger constant} by
$$
\mt (X)=\sup\left\{\be  \in (0, 1] \mid  \inf_{PSH(X, \omega)\cap L^\infty(X)} F_{\omega, \beta}(\cdot)> - \infty \right\}.
$$

\end{definition}

It is straightforward to verify that the invariant  $\mt(X)$ does not depend on the choice of the K\"ahler metric $\omega\in c_1(X)$.

We also define the following paired Mabuchi $K$-energy for conical K\"ahler metrics first introduced in \cite{Be}. 
Here a conical K\"ahler metric is called  smooth if it a polyhomogenous K\"ahler edge metric defined by Jeffres, Mazzeo and Rubinstein in \cite{JMR} and we let $C^\infty_{D, \beta}(X)$ denote  the set of all smooth conical K\"ahler metrics with cone angle $2\pi \beta$ along the smooth divisor $D$.

\begin{definition}
Let $X$ be a Fano manifold. Suppose $\omega$ and $\omega_\varphi$ are two smooth conical K\"ahler metrics  in $c_1(X)$ with cone angle $2\pi (1-(1-\beta)/m)$ along a smooth simple divisor $D \in |-mK_X|$.  The paired Mabuchi $K$-energy is defined by 
\begin{equation}
\mathcal{M}_{\omega, \beta} (\varphi)= \frac{1}{V} \int_X \log \frac{\omega_\varphi^n}{\omega_n} - \beta (I_\omega-J_\omega)_\omega (\varphi) + \frac{1}{V} \int_X h_\omega (\omega^n - \omega_\varphi^n),
\end{equation}
where $h_\omega$ is the Ricci potential of $\omega$ defined by $\ddbar h_\omega = \Ricc(\omega) - \omega$, and $$I_\omega (\varphi) = \sqrt{-1} \sum_{i=0}^{n-1} \int_X \partial \varphi  \wedge \dbar \varphi \wedge \omega^i \wedge \omega_\varphi^{n-i-1}$$ is the Aubin-Yau $I$-functional.

\end{definition}

It is proved in \cite{Be} that if the conical K\"ahelr-Einstein equation is solvable for the data $(D, \beta)$, both $\FF_{\omega, \beta}$ and $\mathcal{M}_{\omega, \beta}$ are bounded below $PSH(X, \omega)\cap C^\infty(X)_{D,\be}$. Furthermore, if one is bounded below, the other must also be bounded below, and conversely, if either of the functionals is $J$-proper, the Monge-Amp\`ere equation associated to the conical K\"ahler-Einstein equation admits a bounded solution \cite{Be}. Moreover, the solution is a smooth conical K\"ahler-Einstein metric in the sense of \cite{JMR}.

\medskip

\subsection{Pluri-anticanonical system}

In this section, we will remove the assumption on the smooth simple divisor $D$ in \cite{D4} and construct conical K\"ahler-Einstein metrics by deforming the angle along the continuity method. 

The following lemma is an immediate consequence of the adjunction formula and Bertini's Theorem.

\begin{lemma}\label{noholV}
 Let $X$ be a Fano manifold of $\dim X\geq 2$. For any sufficiently large $m\in \mathbb{Z}^+$, there exists a smooth simple divisor $D\in \left| -mK_X \right|$. In particular, 
 $$c_1(D)=(1-m)c_1(X)\mid_D = \frac{1-m}{m} [D]|_D.$$
\end{lemma}
This shows that the smooth simple divisor $D$ lies in $ |-mK_X|$ is a Calabi-Yau manifold if and only if $m=1$.

%
%
%\begin{lemma} If  $D$ is a smooth simple divisor in the linear system $|-mK_X|$ for some $m\in \mathbb{Z}^+$,  then $c_1(D)=c_1($D$ is either a K\"ahler manifold with $c_1(D)=0$ when $m=1$ or a K\"ahler manifold with $c_1(X) <0$ when $m\geq 2$.
%
%\end{lemma}
%
%\begin{proof} check
%
%
%\end{proof}

\begin{theorem}\label{nohol}  Let $X$ be a Fano manifold of $\dim X\geq 2$ and $D$ be a smooth simple divisor in $|-mK_X|$ for some $m\in \mathbb{Z}^+$. Then there  does not exist any holomorphic vector field  tangential to $D$.
\end{theorem}

\begin{proof}
For $\dim X=2$ it follows from the classification  that if $X$ admits a nontrivial holomorphic vector field, $X$ is isomorphic to $\PP^2$, $\PP^1\times \PP^1$ or $\PP^2$ blown up at $1,2$ or $3$ points.  $D\in |-mK_X|$ implies that $g(D)\geq 1$. For $\PP^2$ blown up at  $0, 1, 2$ or $3$ points, any holomorphic vector field on $X$ is  the lifting of a holomorphic vector field on $\PP^2$ fixing the blown-up points. Hence any smooth invariant divisor of such holomorphic vector fields on $X$ must be $\PP^1$ with $g(D)=0$. It is also straightforward to check that any invariant divisor on $\PP^1\times \PP^1$ must also be $P^1$. So from now on let us assume that $\dim X\geq 3$.

First, we claim  that if there exists such a holomorphic vector field,  it must vanish along $D$.  Since $X$ is Fano, we have $\pi_1(X)=0$ and hence $\pi_1(D)=0$ by Lefschetz  hyperplane theorem and our assumption that $\dim X\geq 3$.  Since $m>0$,  either $c_1(D)<0$ or $D$ is a simply connected Calabi-Yau manifold by Lemma \ref{noholV}. In both cases $D$ does not admit any nontrivial holomorphic vector field. So our claim is proved.

Second, we claim there is no holomorphic vector field vanishing along $D$.  It suffices to show that  
\begin{equation}\label{h0}
H^0(X,TX\otimes K_X^ m)=0 
\end{equation}
thanks to the following  exact sequence
$$
\begin{CD}
0@>>>TX\otimes  K_X^{m}@>>>TX@>>>TX\mid_D@>>>0.
\end{CD}
$$
Since $TX\otimes K_X\cong \Omega^{n-1}_X$, we have
$$H^0(X,TX\otimes K_X^{ m})=H^0(X,\Omega^{n-1}_X\otimes K_X^{\otimes(m-1)})).$$
Now if $m>1$ then the right hand side is $0$ by
 Kodaira-Akizuki-Nakano vanishing theorem and the fact that $K_X$ is negative. For $m=1$, equation (\ref{h0})   follows from $H^0(X,\Omega^{n-1}_X)\cong H^{n-1}(X,\sO_X)=0$ by Kodaira vanishing theorem and $X$ being Fano.
 \end{proof}

 \begin{remark}
 
Theorem \ref{nohol} is speculated by Donaldson \cite{D4} and is proved in  \cite{Be} for the special case when the holomorphic vector field is Hamiltonian and $m=1$.
 \end{remark}

Combined with the openness result in \cite{D4}, we immediately have the following corollary.

\begin{corollary} \label{open} Let $X$ be a Fano manifold and $D\in |-mK|$ be a smooth simple divisor for some $m\in \mathbb{Z}^+$. If there exists a smooth conical K\"ahler-Einstein metric satisfying $$\Ricc(g) = \beta g + \frac{1-\beta}{m} [D]$$ for some $\beta\in (0,1)$, then there exists $\epsilon >0$ such that for any $\beta'$ with $|\beta-\beta'|< \epsilon$, there exists a smooth conical K\"ahler-Einstein metric $g'$ satisfying $\Ricc(g') = \beta' g + \frac{1-\beta'}{m} [D].$

\end{corollary}

%%%%%%%%%%%%%%%%%%%%%%%%%%%%%%%%%%%

\bigskip

\subsection{The $\alpha$-invariant and the Moser-Trudinger inequality}

Let $X$ be a Fano manifold and $D$ be a smooth  simple divisor in $|-mK_X|$ for some $m\in \mathbb{Z}^+$.  Let $\omega' \in c_1(X) $ be a smooth K\"ahler form and let $\Omega_{\omega'} $ be a smooth volume form on $X$ such that $$\Ricc(\Omega_{\omega'}) = -\ddbar\log \Omega_{\omega'}=  \omega' .$$

We now apply the continuity method and consider the following family of equations for $\beta\in [0, 1]$.
\begin{equation}\label{cont}
(\omega' + \ddbar\varphi_t)^n = e^{-t\varphi_t}  ( \Omega_{\omega'} )^\beta (\Omega_D)^{1-\beta}, ~~~t\in [0, \beta].
\end{equation}
 We let $$S=\{ t\in [0, \beta]\mid \text{(\ref{cont}) is solvable for some } t \text{ with } \omega_t \text{  a smooth conical K\"ahler metric } \}.$$
By the results in \cite{JMR}, $0\in S$ and $S$ is open.  Let $\omega_t = \omega + \varphi_t$ for any $t\in S$. The curvature equation of (\ref{cont}) is given by 
$$\Ricc(\omega_t) = t\omega_t + (\beta-t)\omega' + \frac{1-\beta}{m}  [D] \geq t \omega_t.$$
Hence the Green's function for $\omega_t$ is uniformly bounded below by $t$ for all $t\in S$ \cite{JMR}. Furthermore, let $\Delta_t$ be the Laplace operator associated to $\omega_t$. Then $$\Delta_t \dot \varphi_t = -\varphi_t - t \dot \varphi_t.$$

Following the argument for smooth case with slight modification to the conical K\"ahler metrics, one can show the following proposition.  It is proved in a more general setting in \cite{Be}.

\begin{proposition}\label{lowbd1} Let $X$ be a Fano manifold and $D\in |-mK_X|$ be a smooth simple divisor.

\begin{enumerate}

\item If there exists $\beta \in (0, 1]$ and a smooth conical K\"ahler-Einstein metric $\omega_{KE}$ satisfying $$\Ricc(\omega_{KE}) = \beta \omega_{KE} + \frac{1-\beta}{m} [D], $$ then the paired $F$-functional
$$\FF_{\omega_{KE}, \beta} (\varphi)= J_{\omega_{KE}} (\varphi) - \frac{1}{V} \int_X \varphi \omega_{KE} ^n - \frac{1}{\beta} \log \frac{1}{V} \int_X \left( e^{-\varphi } \Omega_{\omega_{KE}} \right)^\beta ( \Omega_D)^{1-\beta} .$$
is uniformly bounded below for all $\varphi \in PSH(X, \omega_{KE} ) \cap L^\infty(X)$.

\medskip

\item If $\omega\in c_1(X)$ is a smooth K\"ahler metric and 
$\FF_{\omega, \beta} (\varphi) $  is $J$-proper on $PSH(X, \omega) \cap L^\infty(X)$ for some $\beta\in (0, 1]$,  then there exists a unique  smooth conical K\"ahler metric $\omega_{KE} $ solving $$\Ricc(\omega_{KE}  ) = \beta \omega_{KE}  + \frac{1-\beta}{m} [D].$$

\end{enumerate}

\end{proposition}

Same argument in the proof of Proposition \ref{lowbd1} can be applied to prove the following lemma if replacing $m^{-1} [D]$ by a smooth K\"ahler metric $\theta \in c_1(X)$.

\begin{lemma}\label{lowerbd2}   Let $X$ be a Fano manifold and $\theta$ be a smooth K\"ahler metric in $c_1(X)$. 

\begin{enumerate}

\item If there exists  a smooth  K\"ahler metric $\omega_\theta$ on $X$ satisfying $$\Ricc(\omega_\theta) = \beta \omega_\theta + (1-\beta)\theta $$ for some $\beta\in (0, 1]$. Then
$$F_{\omega_\theta, \theta, \beta} (\varphi) = J_{\omega_\theta}(\varphi) - \frac{1}{V} \int_X \varphi \omega_\theta^n - \frac{1}{\beta} \log \frac{1}{V} \int_X e^{-\beta\varphi} (\Omega_{\omega_\theta})^\beta (\Omega_\theta)^{1-\beta}$$
is uniformly bounded below on $PSH(X, \omega) \cap L^\infty(X)$.

\medskip

\item If $\omega\in c_1(X)$ is a smooth K\"ahler metric and 
$F_{\omega, \theta, \beta} (\varphi) $
 is $J$-proper on $PSH(X, \omega) \cap L^\infty(X)$ for some $\beta\in (0, 1]$,  then there exists a unique  smooth  K\"ahler metric $\omega_\theta $ solving $$\Ricc(\omega _\theta ) = \beta \omega_\theta  + (1-\beta)\theta.$$

\end{enumerate}

\end{lemma}

The $\alpha$-invariant is introduced by Tian \cite{T87} to obtain a sufficient condition for the existence of K\"ahler-Einstein metrics on Fano manifolds.  It is shown by Demailly \cite{De} that the $\alpha$-invariant coincides with the log canonical threshold in birational geometry. It is natural to relate the log canonical threshold for pairs to the paired $\alpha$-invariant. It is first introduced in \cite{Be} as a generalization for the $\alpha$-invariant.

\begin{definition}Let $X$ be a Fano manifold and $D\in |-mK_X|$ be a  smooth simple divisor. Let $s$ be the defining section of $[D]$ and $h$ be a smooth hermitian metric for $-mK_X$. Let $\omega \in c_1(X)$ be a smooth K\"ahler metric. Then we define the paired $\alpha$-invariant for $\beta \in (0,1]$ by
\begin{equation}
\alpha_{D, \beta} (X)= \sup \left\{ \alpha>0 ~\left|~ \sup_{ \varphi \in PSH(X, \omega)\cap L^\infty(X) }  \int_X |s|_h^{-\frac{2(1-\beta)}{m} } e^{-\alpha\beta({\varphi-\sup\varphi})} \omega^n <\infty    \right.   \right\}.
\end{equation}

\end{definition}

It is straightforward to check that the invariant $\alpha_{D, \beta}$ does not depend on the choice of $\omega \in c_1(X)$.   The following theorem is due to Berman \cite{Be} for an effective bound on $\alpha_{D, \beta}$ to construct conical K\"ahler-Einstein metrics by combining the results in \cite{JMR}.

\begin{theorem}\label{al} There exists $\beta_D \in (0, 1]$ such that for all $\beta\in (0,\beta_D]$ we have
\begin{equation}
\alpha_{D, \beta} (X)> \frac{n}{n+1} .
\end{equation}
In particular, there exists a smooth conical K\"ahler-Einstein metric $\omega \in c_1(X)$ satisfying $$\Ricc(\omega) = \beta \omega + \frac{ 1-\beta }{m} [D] $$ for $\beta \in (0, \beta_D)$. 

\end{theorem}

In \cite{S04}, the first author proves that if the $\alpha$-invariant on an $n$-dimensional Fano manifold is greater $n/(n+1)$, then the $F$-functional is $J$-proper. The following theorem is a generalization for the conical case.

\begin{theorem}\label{alprop} Let $X$ be a Fano manifold and $\omega\in c_1(X)$ be a smooth K\"ahler metric. If  $D\in |-mK_X|$ is a smooth simple divisor and if there exists $\beta\in (0, 1]$ such that 
$$\alpha_{D, \beta} (X)> \frac{n}{n+1},$$
  then the functional 
  $$\FF_{\omega, \beta} (\varphi)= J_{\omega} (\varphi) - \frac{1}{V} \int_X \varphi \omega^n - \frac{1}{\beta} \log \frac{1}{V} \int_X \left( e^{-\varphi } \Omega_{\omega} \right)^\beta ( \Omega_D)^{1-\beta} $$
as in Definition \ref{modf},  is $J$-proper on $PSH(X, \omega)\cap L^\infty(X)$.
\end{theorem}

\begin{proof} We break the proof into the following two steps.

\begin{enumerate}

\item[Step 1.]  Since $\alpha_{D, \beta}(X)>n/(n+1)$, by Theorem \ref{al} there exists a smooth conical K\"ahler-Einstein metric $\omega_{KE}$ satisfying 
$$\Ricc(\omega_{KE}) = \beta \omega_{KE} + \frac{1-\beta}{m} [D].
$$
Let $PSH(X, \omega_{KE}, K)$ be the set of all $\varphi \in PSH(X, \omega_{KE})\cap L^\infty(X) $ such that
\begin{equation}\label{assum}
\osc_X \varphi = \sup_X \varphi - \inf_X \varphi \leq (n+1)J_{\omega_{KE}} (\varphi) +K.
\end{equation}
We claim that  $\FF_{\omega_{KE}, \beta}$ is $J$-proper for all $\varphi\in PSH(X, \omega_{KE}, K)$. To see that, take $\al$ satisfying 
$$\frac{n\beta}{n+1}<\al< \beta \min (\al_{D,\be},  1)$$ and let $\Omega_D=|s|^{2/m}$ then we have
\begin{eqnarray*}
\int_Xe^{-\be\varphi}\Omega_\be
&=&\int_Xe^{-\al(\varphi-\sup\varphi)+(\al-\be)\phi-\al\sup\varphi} \Omega_D\\
&\leq&C e^{(\al-\be)\inf\varphi-\al\sup\varphi} \int_X e^{-\al(\varphi-\sup \varphi)}\Omega_D\\
&\leq&C e^{(\al-\be)\inf\varphi-\al\sup\varphi}
\end{eqnarray*}
by the definition of $\alpha_{D, \beta}$.
By assumption \eqref{assum}, we have
\begin{eqnarray*}
\int_Xe^{-\be\varphi}\Omega_D
&\leq &Ce^{(\be-\al)(n+1)J_{\omega_{KE}}(\varphi)-\be\sup \varphi)} \\
&\leq &Ce^{(n+1)(\be-\al)J_{\omega_{KE}}(\varphi)-\frac{\be}{V}\int_X \varphi\omega_{KE}^n} \\
&=&Ce^{\be J_{\omega_{KE}}(\varphi)-((n+1)\al-n\be)J_{\omega_{KE}}(\varphi)-\frac{\be}{V}\int_X \varphi\omega_{KE}^n},
\end{eqnarray*}
and our claim follows by taking logarithm of both sides of the above inequality.

\medskip

\item[Step 2.]  Now we will remove the assumption (\ref{assum}) for $\varphi\in PSH(X, \omega)\cap L^\infty(X)$.  We first consider all $\varphi \in PSH(X, \omega)$ such that $\omega'= \omega_{KE}+\ddbar \varphi$ is a smooth conical K\"ahler metric with cone angle $2\pi (1 - \frac{1-\beta}{m})$ along $D$. By applying the same argument for smooth K\"ahler-Einstein metrics to solve the continuity method backward, one can show that
$$(\omega' + \ddbar \varphi_t ) ^n = \left( e^{- t \varphi_t } \Omega_{\omega'} \right)^{\beta} \Omega_D^{1-\beta}$$ admits a smooth conical solution for $t\in [0, 1]$ because the implicit function theorem can be applied at $t=1$ due to Theorem \ref{nohol}. In particular, $\varphi_1 = - \varphi.$

Let $\omega_t = \omega' + \ddbar \varphi_t$. Note that for $t\geq 1/2$, the Ricci curvature of $\omega_{KE}+ \ddbar (\varphi_t - \varphi_1)= \omega_t$ is no less than $\beta/2$. Then the Green's functions for both $\omega$ and $\omega_t$ are uniformly bounded from below by $-G$ for some positive number $G$, and $$ \Delta_{\omega_{KE}} (\varphi -\varphi_1) = tr_{\omega_{KE}} (\omega_t - \omega_{KE}) \geq -n , ~~ \Delta_{\omega_t} (\varphi_t - \varphi_1) \leq n. $$  Then by the Green's formula,  for $t\geq \beta/2$, we have
$$ \frac{1}{V} \int_X ( \varphi_t - \varphi_1) \omega_{KE}^n - nG  \leq   ( \varphi_t - \varphi_1)  \leq \frac{1}{V} \int_X ( \varphi_t - \varphi_1) \omega_t^n + n G.$$
Hence $$\osc_X (\varphi_t - \varphi_1) \leq I_{\omega_{KE}} (\varphi_t - \varphi_1) + 2nG \leq (n+1)J_{\omega_{KE}}(\varphi_t- \varphi_1) + 2nG. $$
This implies that $\varphi_t - \varphi_1\in PSH(X, \omega_{KE}, 2n G)$ and then  the $J$-properness holds for $\varphi_t - \varphi_1$, and there exist $\delta, C_1, C_2>0$ such that
\begin{eqnarray*}
\FF_{\omega_{KE}, \beta} (\varphi_t - \varphi_1) &\geq&  \delta J_{\omega_{KE}} (\varphi_t - \varphi_1) - C_1\\
&\geq& \frac{\delta}{n+1} \osc_X (\varphi_t - \varphi_1) -C_2.
\end{eqnarray*}
Consequently, there exist $C_3>0$ such that
\begin{eqnarray*}
n (1-t) J_{\omega_{KE}} (\varphi) &\geq & (1-t)  J_{\omega'}(\varphi_1)\\
&\geq & (1-t) ( I _{\omega'} (\varphi_1) - J_{\omega'}(\varphi_1))\\
&\geq & \int_t^1(I_{\omega'}(\varphi_s) - J_{\omega'}(\varphi_s)) ds\\
&=& \FF_{\omega', \beta}( \varphi_t) - \FF_{\omega', \beta} (\varphi_1) \\
&=& \FF_{\omega_{KE}, \beta} (\varphi_t - \varphi_1) \\
&\geq & \left(\frac{\al}{\be}-\frac{n}{n+1}\right) \osc_X (\varphi_t - \varphi_1) - C_3.
\end{eqnarray*}
The third inequality follows by the increasing monotonicity for $(I_{\omega'} - J_{\omega'})(\varphi_t) $ for $t\in [0, 1]$. Then by applying the cocycle condition and  the same argument in the smooth case  in \cite{T97,TZ}, we have
\begin{eqnarray*}
\FF_{\omega_{KE}, \beta} (\varphi) &= & - \FF_{\omega', \beta} (\varphi_1)\\
&=& \int_0^1 ( I_{\omega'} (\varphi_t) - J_{\omega'} (\varphi_t)) dt \\
&\geq& (1-t) (I_{\omega'} (\varphi_t) - J_{\omega'}(\varphi_t)) \\
&\geq& \frac{1-t}{n} J_{\omega'} (\varphi_t) \\
&\geq& \frac{1-t}{n} J_{\omega'}(\varphi_1) -\frac{2 (1-t)}{n} \osc_X (\varphi_t - \varphi_1) - C_4\\
&\geq&\frac{ 1-t }{n^2} J_{\omega_{KE}} (\varphi) - C_5 (1-t)^2J_{\omega_{KE}} (\varphi) - C_6.
\end{eqnarray*}

Since $C_5$ and $C_6$ are independent of the choice for $t\geq 1/2$, by choosing $t$ sufficiently close to $1$, we can find $\epsilon', C_{\epsilon'}>0$, such that
\begin{equation}\label{FF-ppr}
\FF_{\omega_{KE}, \beta}(\varphi) \geq \epsilon' J_{\omega, \beta}(\varphi) - C_{\epsilon'}
\end{equation}
for all $\varphi \in PSH(X, \omega_{KE})$ such that $\omega_\varphi$ is a smooth conical K\"ahler metric with cone angle $2\pi (1-\frac{1-\beta}{m})$ along $D$. The set of such $\varphi$ is dense in $PSH(X, \omega_{KE}) \cap L^\infty(X)$ and so the $J$-properness holds for $PSH(X, \omega_{KE})\cap L^\infty(X)$.

\medskip
\item[Step 3.]Finally, the $J$-properness for any $\omega$ follows from Lemma \ref{cocy} and the fact  the Green's function for $\omega_{KE}$ is bounded from below.

\end{enumerate}

\end{proof}

%%%%%%%%%%%%%%%%%%%%%%%%%%%%%%%%%%%%%%%%%

\subsection{An interpolation formula}

In this section, we will prove the following  interpolation formula for the $F$-functional to obtain $J$-properness.

\begin{proposition}\label{proper4.4}  Let $X$ be a Fano manifold and $D$ a smooth simple divisor in $|-mK_X|$ for some  $m\in \mathbb{Z}^+$. Let $\omega$ be a smooth K\"ahler metric in $c_1(X)$.  If there exists $\alpha \in (0, 1]$ such that
$$
\inf_{PSH(X, \omega)\cap L^\infty(X)} \FF_{\omega, \alpha} (\cdot ) >-\infty
$$
 then
 $\FF_{\omega, \beta} (\varphi) $
 is $J$-proper on $PSH(X, \omega) \cap L^\infty(X)$ for all $\beta \in (0, \alpha)$.

\end{proposition}

\begin{proof}   We want to show that
$$\FF_{\omega, \beta} (\varphi) = J_{\omega} (\varphi) - \frac{1}{V} \int_X \varphi \omega^n - \frac{1}{\beta} \log \frac{1}{V} \int_X  (e^{-\varphi}\Omega_\omega)^{\beta} (\Omega_D)^{1-\beta}$$
 is $J$-proper for all $\beta\in (0,  \alpha)$.

First for $0<\tau<\be<\al$, we write  $\be=\tau/p+\al/q$ for some $1/p+1/q=1$.  Then  H\"{o}lder inequality implies  the  interpolation
\begin{eqnarray*}
\FF_{\omega,\be}(\varphi)\!\!\!\!\!
 &=& \!\!\!J_\omega(\varphi)-\frac{1}{\be}\log\frac{1}{V}\int_X(e^{-\varphi}\Omega_\omega)^\be \Omega_D^{1-\be}\\
&=&\!\!\!\left( \frac{\tau}{\be p}+\frac{\al}{\be q}\right) J_\omega(\varphi)-\frac{1}{\be}\log\frac{1}{V}\int_X\left(e^{-\varphi}\frac{\Omega_\omega}{\Omega_D}\right)^{\tau/p+\al/q} \cdot \Omega_D\\
%\text{( Since  $\int_X\Omega_D=V$ )}
&\geq &\!\!\!
\frac{\tau}{\be p}\left( J_\omega(\varphi)-\frac{1}{\tau}\log\frac{1}{V} \int_X(e^{-\varphi}\Omega_\omega)^\tau\Omega_D^{1-\tau} \right)
+\frac{\al}{\be q}\left(J_\omega(\varphi)-\frac{1}{\al}\log \frac{1}{V}\int_X (e^{-\varphi}\Omega_\omega)^\tau\Omega_D^{1-\tau}\right)\\
&=&\frac{\tau}{\be p}\cdot \FF_{\omega,\tau}(\varphi)+\frac{\alpha}{\be q}\cdot \FF_{\omega,\al}(\varphi)\\
&\geq&\frac{\tau}{\be p}\cdot \FF_{\omega,\tau}(\varphi) - C_1\\
&\geq& \epsilon J_{\omega} (\varphi)-C.
\end{eqnarray*}
The last inequality follows from Theorem \ref{alprop} by choosing $\tau$ sufficiently small so that $\alpha_{D, \tau} > n/(n+1)$.

\end{proof}

The same argument in the proof of Proposition \ref{proper4.4} can be applied to prove the following lemma after replacing $m^{-1}[D]$ by a smooth K\"ahler metric $\theta\in c_1(X)$.

\begin{lemma}\label{prop4.5}  Let $X$ be a Fano manifold and $\theta$ be a smooth K\"ahler metric in $c_1(X)$. Let $\omega$ be a smooth K\"ahler metric in $c_1(X)$.  If there exists $\alpha \in (0, 1]$ such that
$$
F_{\omega, \alpha} (\varphi) = J_{\omega} (\varphi) - \frac{1}{V} \int_X \varphi \omega^n - \frac{1}{\alpha} \log \frac{1}{V} \int_X e^{-\alpha\varphi} (\Omega_\omega)^{1-\alpha} (\Omega_\theta)^{1-\alpha}
$$
 is bounded  below on $PSH(X, \omega) \cap L^\infty(X)$, then
 $F_{\omega, \beta} (\varphi) $
 is $J$-proper  on $PSH(X, \omega) \cap L^\infty(X)$ for all $\beta\in (0, \alpha)$.

\end{lemma}

We remark that Lemma \ref{prop4.5}  also serves as an alternative proof for Theorem 1.1  in \cite{Sze} relating $R(X)$ and the continuity method.

\bigskip

\subsection{Proof of Theorem \ref{main1}}

Let us prove the first part of Theorem \ref{main1}.

\begin{proposition} \label{alphaup}Let $X$ be a Fano manifold and $D$ be a smooth simple divisor in $|-mK_X|$ for some  $m\in \mathbb{Z}^+$. If there exist $\beta\in (0, 1]$ and  a smooth conical K\"ahler-Einstein metric $\omega$ satisfying $$\Ricc(\omega) = \beta \omega + \frac{1-\beta}{m} [D], $$
then $$\beta\leq R(X).$$ In particular, the inequality holds if and only if  $\beta=1$.

\end{proposition}

\begin{proof} Let $\omega_{KE}$ be a smooth conical K\"ahler-Einstein metric on $X$ satisfying $\Ricc(\omega_{KE}) = \beta \omega_{KE} + \frac{1-\beta}{m} [D]$. By  Proposition \ref{lowbd1}, we know that  $\FF_{\omega_{KE}, \beta}$ is bounded below.
%%%%%
By Proposition \ref{proper4.4}, $\FF_{\omega, \beta'}$ is $J$-proper for all $\beta' \in (0, \beta)$.

 Let $\omega, \theta \in c_1(X)$ be two smooth K\"ahler metrics on $X$. The $J$-properness of $\FF_{\omega, \beta'}$ immediately implies the $J$-properness for
$$F_{\omega, \theta, \beta'} (\varphi) = J_{\omega} (\varphi) - \frac{1}{v} \int_X \varphi \omega^n - \frac{1}{\beta'} \log \frac{1}{V} \int_X e^{-\beta' \varphi} (\Omega_\omega)^{\beta'} (\Omega_\theta)^{1-\beta'}$$
 because $\Omega_D$ is strictly bounded below from $0$.
Then by Lemma \ref{lowerbd2}, there exists a minimizer of $F_{\omega, \theta, \beta'}$ which soloves the equaiton
$$\Ricc(\omega)= \beta' \omega + (1-\beta') \theta \geq \beta' \omega .$$
This shows that $R(X)\geq \beta'$ and so $R(X) \geq \beta$.

If $\beta= R(X) <1$, there must exist $\epsilon>0  $ and a smooth conical K\"ahler-Einstein metric $g$ such that $\Ricc(g) = (\beta+\epsilon) g + (1-\beta-\epsilon)m^{-1} [D]$ by Corollary \ref{open}. Then it is a contradiction to the definition of $R(X)$ by repeating the previous argument.

\end{proof}

We now prove the second part of Theorem \ref{main1}.

\begin{proposition}\label{alphalow} Let $X$ be a Fano manifold. Then for any $\beta \in (0, R(X))$, there exist a smooth simple divisor $D\in |-mK_X|$ for some $m\in \mathbb{Z}^+$ and a smooth conical K\"ahler-Einstein metric $g$ satisfying
$$\Ricc(\omega) = \beta \omega + \frac{1-\beta}{m} [D].$$

\end{proposition}

\begin{proof} We break the proof into the following steps.

\begin{enumerate}

\item[Step 1.] Let $\omega$ and $\theta$ be two smooth K\"ahler metrics in $c_1(X)$. For any $\beta \in (0, R(X))$, by Szeklyhidi's result \cite{Sze}, the following family of Monge-Amp\`ere equation
$$
(\omega+\ddbar \varphi_t)^n = \left( e^{- t \varphi } \Omega_{\omega} \right) ^{\beta} (\Omega_\theta)^{1-\beta}, ~~~ \int_X (\Omega_{\omega})^\beta (\Omega_\theta)^{1-\beta }  = V$$
is solvable for all $t\in [0, 1]$. Then by Lemma \ref{lowerbd2},
$$F_{\omega, \beta} (\varphi) = J_{\omega} (\varphi) - \frac{1}{V} \int_X \varphi \omega^n - \frac{1}{\beta} \log \int_X e^{-\beta \varphi} (\Omega_{\omega})^\beta (\Omega_\theta)^{1-\beta}$$
is bounded below on $PSH(X, \omega) \cap L^\infty(X)$. By Lemma \ref{prop4.5}, for any $\beta' \in (0, \beta)$, $F_{\omega, \beta'} (\varphi) $ is $J$-proper. 
It immediately follows that  for any $\beta\in (0, R(X))$, there exist $\epsilon, C_\epsilon>0$ such that    for all $\varphi \in PSH(X, \omega) \cap L^\infty(X)$,
$$\int_X e^{- \beta \varphi} \omega^n \leq C_\epsilon e^{ (\beta -\epsilon) J_{\omega}(\varphi) - \frac{\beta}{V} \int_X \varphi \omega^n }. $$

\item[Step 2.] Let  $D$ be a smooth simple divisor in $|-mK_X|$. We will later choose $m$ sufficiently large.  Let $s$ be a defining section of $D$ and $h$ be a smooth hermitian metric on $-mK_X$.  For any $\beta \in (0, R(X))$,  there exists $\delta>0$ such that $\beta + \delta < R(X)$.
Then
\begin{eqnarray*}
\int_X |s|_h^{-\frac{2(1-\beta)}{m} }e^{ - \beta \varphi} \omega^n&\leq &   \left(\int_X e^{-(\beta+\delta) \varphi} \omega^n \right)^{\frac{\beta}{\beta+\delta}}  \left( \int_X |s|_h^{-\frac{2(1-\beta)(\beta+\delta) }{m\delta} } \omega^n \right)^{\frac{\delta}{\beta+\delta}}\\
&\leq& C_\delta \left(\int_X e^{-(\beta+\delta) \varphi} \omega^n \right)^{\frac{\alpha}{\alpha+\delta}}
\end{eqnarray*}
if we choose $m > \frac{\delta}{(1-\beta)(\beta+ \delta)}$.
By the conclusion in Step 1, there exist $\epsilon, C_\epsilon>0$ such that
$$\int_X |s|^{-\frac{2(1-\beta)}{m} }e^{ - \beta \varphi} \omega^n \leq C_\epsilon e^{ (\beta -\epsilon) J_{\omega}(\varphi) - \frac{\beta}{V} \int_X \varphi \omega^n }. $$
Equivalently,
$\FF_{\omega, \beta}(\varphi)$ is $J$-proper on $PSH(X, \omega)\cap L^\infty(X)$.
By Proposition \ref{lowbd1}, there exists  a unique smooth conical K\"ahler-Einstein metric $\omega_\beta$ solving $$\Ricc(\omega_\beta) = \beta \omega_\beta + \frac{1-\beta}{m} [D].$$

\end{enumerate}

\end{proof}

Now we can relate the optimal Moser-Tridinger constant to the invariant $R(X)$. 

\begin{corollary} \label{flowbd} Let $X$ be a Fano manifold and $\omega\in c_1(X)$ be a smooth K\"ahler metric.

\begin{enumerate}

\item If $\beta \in (0, R(X))$,  $F_{\omega, \beta} $ is $J$-proper on $PSH(X, \omega) \cap L^\infty(X)$. Equivalently, there exist $\epsilon, C_\epsilon>0$ such that the following Moser-Trudinger inequality holds for all $\varphi \in PSH(X, \omega)\cap L^\infty(X)$
$$\int_X e^{-\beta \varphi} \omega^n \leq C_\epsilon e^{ (\beta-\epsilon) J_\omega (\varphi) - \frac{\beta}{V} \int_X \varphi \omega^n}. $$

\item If $\beta\in (R(X), 1)$, then $$\inf_{PSH(X, \omega)\cap L^\infty(X)} F_{\omega, \beta}(\cdot) = - \infty. $$

\end{enumerate}

\end{corollary}

\begin{proof} For $\beta\in (0, R(X))$ and a fixed smooth K\"ahler metric $\theta\in c_1(X)$, there exists a smooth K\"ahler metric $\omega$ satisfying $\Ricc(\omega) = \beta\omega + (1-\beta) \theta$. The corollary is an immediate consequence of  by modifying  the interpolation formula in Proposition \ref{proper4.4}, after replacing $m^{-1} [D]$ by  $\theta$.

\end{proof}

Immediately, we can show that $R(X)$ and $\mt(X)$ take the same value for a Fano manifold $X$. 
\begin{corollary}

Let $X$ be a Fano manifold. Then 
\begin{equation}
\mt(X) = R(X) = \sup \{ \beta\in (0, 1) ~|~F_{\omega, \beta}~\text{is~} J\text{-proper~on } PSH(X, \omega)\cap L^\infty(X)\}, 
\end{equation}
where $\omega\in c_1(X)$ is a smooth K\"ahler metric.

\end{corollary}

\bigskip

%%%%%%%%%%%%%%%%%%%%%%%%%%%%%%

\subsection{Proof of Theorem \ref{main2}}

Before proving Theorem \ref{main2}, we first quote the following proposition establishing the equivalence for  the Mabuchi $K$-energy and the $F$-functional when either of them is  bounded below proved in \cite{LiH} by applying the K\"ahler-Ricci flow and Perelman's estimates for the scalar curvature.

\begin{proposition}\label{knu} Let $X$ be a Fano manifold and $\omega\in c_1(X)$ be a smooth K\"ahler metric. Then Ding's functional $F_\omega$ is bounded below on $PSH(X, \omega)\cap C^\infty(X)$ if and only if the Mabuchi $K$-energy is bounded below.

\end{proposition}

Proposition \ref{knu} holds for the paired Mabuchi $K$-energy and the paired $F$-functional as shown in \cite{Be}.  One can also apply the continuity method for the conical K\"ahler metrics with positive Ricci curvature as in \cite{Ru}.  Let $PSH(X, \omega) \cap C_{D, \beta}^\infty(X)$ be the set of all bounded $\varphi$ such that $\omega + \ddbar \varphi$ is a smooth conical K\"ahler metric with cone angle $2\pi\beta$ along $D$.

\begin{proposition}\label{knu2} Let  $X$ be a Fano manifold and $D\in |-mK_X|$ be a smooth simple divisor. Let $\omega$ be a smooth conical K\"ahler metric in $c_1(X)$ with cone angle $2\pi (1-\beta)/m$ along $D$. Then  $$ \inf_{ PSH(X, \omega)\cap C^\infty_{D, \frac{1-\beta}{m} } (X) }\mathcal{M}_{\omega, \beta} (\cdot)> -\infty$$ is equivalent to 
$$\inf_{PSH(X, \omega) \cap L^\infty(X)} \mathcal{F}_{\omega, \beta} (\cdot) >     t - \infty.$$

\end{proposition}

We can now prove Theorem \ref{main2}.

\begin{theorem} \label{main4_1}Let $X$ be a Fano manifold and $D$ be a smooth simple divisor in $|-mK_X|$ for some $m\in \mathbb{Z}^+$. If the paired Mabuchi $K$-energy $\mathcal{M}_{\omega, R(X)} $ on $X$ is bounded below, then for any $\beta\in (0, R(X))$ there exists a smooth conical K\"ahler-Einstein metric satisfying
\begin{equation}\label{keeqn4}
\Ricc(g) = \beta g + \frac{1-\beta}{m}[D].
\end{equation}

\end{theorem}

\begin{proof}

Let $\omega\in c_1(X)$ be a smooth K\"ahler metric. Then by Proposition \ref{knu2},
$\mathcal{F}_{\omega, R(X)}(\varphi)$  %
 is bounded below on $PSH(X, \omega)\cap L^\infty(X)$. Applying the interpolation formula in Proposition  \ref{proper4.4}, $\FF_{\omega, \beta}$ is $J$-proper on $PSH(X, \omega) \cap L^\infty(X)$ for all $\beta\in (0, R(X))$. The theorem follows by Proposition \ref{lowbd1}.
\end{proof}

When the Mabuchi $K$-energy is bounded below on $X$, for any $\beta\in (0, 1)$, there exists a conical K\"ahler-Einstein metric satisfying equation (\ref{keeqn4}) for $m=1$. In this case, $D$ is a smooth Calabi-Yau hypersurface of $X$. If we only assume $R(X)=1$, we have the same conclusion in Theorem \ref{main4_1} for the linear systems $|-mK_X|$ with $m\geq 2$.

\begin{proposition} Let $X$ be a Fano manifold and $D$ be a smooth simple divisor in $|-mK_X|$ for some $m\geq 2$. Then for any $\beta\in \left(0, \frac{(m-1)\cdot R(X)}{m-R(X)}\right)$, there exists a smooth conical K\"ahler-Einstein metric $\omega$ satisfying $$\Ricc(\omega) = \beta \omega + \frac{1-\beta}{m} [D].$$ In particular, when $R(X)=1$, we have conical K\"ahler-Einstein metric for any $\be\in (0,1)$.

\end{proposition}

\begin{proof}  We prove the case for $R(X)=1$ and the general case follows from the same argument. Let $s$ be the defining section of $|-mK_X|$ and $h$ be a smooth hermitian metric of $-mK_X$. Then $|s|_h^{-(2-\epsilon)}$ is integrable for any $\epsilon>0$. Furthermore, $F_{\omega, (1-\beta)m^{-1}[D], \beta}$ is proper for any smooth K\"ahler forms $\omega, \theta\in c_1(X)$ if $\beta\in (0,1)$, Then the proposition can be proved by similar interpolation argument in the proof of Theorem \ref{main4_1}.

\end{proof}

%%%%%%%%%%%%%%%%%%%%%%%%%%%%%%%%%%%%%%

\bigskip

\section{Conical toric K\"aher-Einstein metrics}

\subsection{Conical toric K\"ahler metrics}

In this section, we will introduce toric conical K\"ahler metrics on toric K\"ahler manifolds and corresponding toric K\"ahler and symplectic potentials as in \cite{D1, D2}. We begin with some basic definitions for toric manifolds.

\begin{definition}

A convex polytope $P$ in $\mathbb{R}^n$  is called a Delzant polytope if a neighborhood of any vertex  of $P$ is  $SL(n,\ZZ)$ equivalent to $\{x_j\geq 0, j=1, ..., n\} \subset \mathbb{R}^n $. $P$ is called an integral Delzant polytope if each vertex of $P$ is a lattice point in $\mathbb{Z}^n$.

\end{definition}

Let $P$ be an integral Delzant polytope in $\mathbb{R}^n$ defined by
\begin{equation}
P=\{ x\in \mathbb{R}^n~|~ l_j (x) >0, j=1, ..., N\},
\end{equation}
where $$l_j(x) = v_j \cdot x +  \lambda_j$$ and  $v_i $ a primitive integral vector in $\mathbb{Z}^n$ and $\lambda_j \in \mathbb{Z}$ for all $j=1, ..., N$. Then $P$ defines an $n$-dimensional nonsingular toric variety by the following observation.

For each $n$-dimensional integral Delzant polytope $P$, as in \cite{D1, D2}, we consider the set of pairs $\{p,  v_{p, i} \}$, where $p$ is a vertex of $P$ and the neighboring faces are given by $l_{p, i} (x)  = v_{p, i} \cdot x - \lambda_{p, i} >0$ for $i=1, ..., n$. For each $p$, we choose a coordinate chart $\mathbb{C}^n$ with $z=(z_1, ..., z_n)$. Then for any two vertices $p$ and $p'$, there exists $\sigma_{p, p'} \in GL(n, \mathbb{Z})$ such that
$$ \sigma_{p, p'} \cdot  v_{p, i} = v_{p', i}. $$
 Furthermore, we have
$$\sigma_{p, p'} \cdot  \sigma_{p', p''} \cdot \sigma_{p'', p} = 1.$$
Therefore $\sigma_{p, p'}$ serves as the transition function for two coordinate charts over $ (\mathbb{C}^*)^n$.

More precisely, let $z=(z_1, ..., z_n)$ and $z'=(z_1', ..., z_n') \in \mathbb{C}^n$ be the coordinates for the chart associated to $p$ and $p'$ respectively. Suppose $\sigma_{p, p'}=(\alpha_{ij})$. Then
$$z_i' = \prod_{j} z_j^{\alpha_{ij}}.$$
Each integral Delzant polytope uniquely determines a nonsingular toric variety $X_P$ by such a construction with the data $(p, \{ v_{p, i}\}_{i=1}^n)$. The constant $\lambda_{p, i}$ determines an ample line bundle $L$ over $X_P$,  and moreover, $$H^0(X_P, L) =span\{ z^{\alpha}\}_{\alpha \in \mathbb{Z}^n \cap \overline{P}}. $$

 Let $\varphi_P = \log (\sum_{\alpha \in \mathbb{Z}^n\cap \overline{P}} |z|^{2\alpha} )$. Then $\omega_P = \ddbar \varphi_P$ is a smooth K\"ahler metric on $(\mathbb{C}^*)^n$ and it can be smoothly extended to a smooth global toric K\"ahler metric on $X_P$ in $c_1(L)$. Then the space of toric K\"ahler metrics in $c_1(L)$ is equivalent to the set of all smooth plurisubharmonic function $\varphi$ on $(\mathbb{C}^*)^n$ such that $ \varphi- \varphi_P$ is bounded and $\ddbar \varphi$ extends to a smooth K\"ahler metric on $X_P$. If we consider the toric K\"ahler potential $\varphi$ which is invariant under the real torus action,  we can view $\varphi$ as a function in $\mathbb{R}^n$ by
 $$\varphi = \varphi(\rho), ~~~\rho=(\rho_1, ..., \rho_n), ~\rho_i = \log |z_i|^2.$$
 One can also define a symplectic potential $u$ on $P$ by
 \begin{equation}
 u(x) = \sum_{j=1}^N l_j (x) \log l_j(x) + f(x)
 \end{equation}
such that $f(x)\in C^\infty(\overline{P})$ and $u(x)$ is strictly convex in $P$.
It is due to Guillemin \cite{Gu} that the toric K\"ahler potential and the symplectic potential are related by the Legendre transform
$$\varphi (\rho) = \mathcal{L} u (\rho) = \sup_{x\in P} (x\cdot \rho - u(x)), ~ u (x) = \mathcal{L} \varphi(x) = \sup_{\rho\in \mathbb{R}^n } (x\cdot \rho - \varphi(\rho))$$
or equivalently
$$\varphi(\rho)= x\cdot \rho - u(x), ~~~ u(x) = x\cdot \rho - \varphi(\rho), ~~~x=\nabla_\rho \varphi(\rho), ~\rho= \nabla_x u(x).$$

Now we would like to generalize the Guillemin condition to toric conical K\"ahler metrics on $X_P$. This can be considered as a generalization of orbifold K\"ahler metrics by replacing the  finite subgroup by possibly infinite non-discrete subgroup of $(S^1)^n$.  Suppose that the integral Delzant polytope is defined by
$$P=\{ x\in \mathbb{R}^n ~|~ l_j(x)=v_j \cdot x + \lambda_j >0, ~j=1, ..., N \}$$ with $v_j \in \mathbb{Z}^n$ being a primitive lattice point  and $\lambda_j \in \mathbb{Z}$.

On each coordinate chart determined by the pair $(p, \{v_{p, i}\}_{i=1}^n)$  associated to a vertex of $P$, we let $z=(z_1, ..., z_n)$ be the coordinates on $\mathbb{Z}^n$. $\{ z_i =0\}$ extends to a smooth toric divisor of $X_P$. Let $D$ be a toric divisor of $X_P$ and suppose $D$ restricted to this coordinate chart is given by $$\sum_{i=1}^n a_i [z_i=0]. $$
For any function $f(z)$  invariant under the $(S^1)^n$-action, we can  lift it to a function
$$ \tilde f(w) = f(z)$$
by letting
$$|w_i| = |z_i| ^{1/ \beta_i}, ~w=(w_1, ..., w_n) \in \mathbb{C}^n, $$
and clearly $\tilde f(w)$ is also $ (S^1)^n$-invariant.  $w\in \mathbb{C}^n$ is then a $\beta$-covering of $z\in \mathbb{C}^n$. Then we can consider the $(S^1)^n$-invariant function space for $k\in \mathbb{Z}^{\geq 0}$ and $\alpha \in [0,1]$
$$C^{k, \alpha}_{\beta, p} = \{ f(z)=f(|z_1|, ..., |z_n|) ~|~ \tilde f(w) \in C^{k, \alpha} (\mathbb{C}^n)  \} .$$
This in turn defines the weighted function space
$$C^{k, \alpha}_{\beta}(X_P), \beta=(\beta_1, ..., \beta_N)\in (\mathbb{R}^+)^N $$ whose restriction on each chart belongs to $C^{k, \alpha}_\beta$ with respect to the weight $\beta$ and $\beta_j$ corresponding to the divisor induced by $l_j(x)=0$.
Then we can define the space of weighted toric K\"ahler metrics on $X_P$ by  considering  a K\"ahler current $\omega$ whose restriction on each chart is given by $$\omega= \ddbar \varphi_p$$ such that  $\varphi_p \in C^{\infty}_{\beta, p}$. Such a weighted toric K\"ahler metric is naturally a smooth conical K\"ahler metric with cone angle $2\pi \beta_i$ along $[z_i =0]$ and is called  a smooth $\beta$-weighted K\"ahler metric.  The local lifting $\tilde \varphi (w)$ is a smooth plurisubharmonic function on the lifting space $w\in \mathbb{C}^n$.

We can also define the space of weighted toric K\"ahler potential $\varphi$ on $(\mathbb{C}^*)^n$ such that $\varphi - \varphi_P$ is bounded and $\ddbar \varphi$ extends to a smooth weighted K\"ahler metric on $X_P$.

We now define a weighted $C^\infty_\beta$ {\em symplectic potential}
$$u(x) = \sum_{j=1}^N\beta_j^{-1} l_j(x) \log l_j(x) + f(x)$$ for $f\in C^\infty(\overline P)$ for $j=1, ..., n$  such that  $f\in C^\infty(\overline{P})$ and $u$ is strictly convex in $P$. Then the weighted K\"ahler potentials and the weighted symplectic potential determine each other uniquely.
The following is a straightforward generalization of the Guillemin condition for conical toric K\"ahler metrics.

\begin{proposition} The weighted $C^\infty_\beta$ toric potential $\varphi$ and the weighted $C^\infty_\beta$ symplectic potential are related by
the Legendre transform

$$\varphi (\rho) = \mathcal{L} u (\rho) = \sup_{x\in P} (x\cdot \rho - u(x)), ~ u (x) = \mathcal{L} \varphi(x) = \sup_{\rho\in \mathbb{R}^n } (x\cdot \rho - \varphi(\rho))$$
or equivalently
$$\varphi(\rho)= x\cdot \rho - u(x) \text{ with } x=\nabla_\rho \varphi(\rho)\text{ and }u(x) = x\cdot \rho - \varphi(\rho)\text{ with } ~\rho= \nabla_x u(x).$$

In particular, if $u(x)=\beta_j^{-1} l_j(x) \log l_j(x) + f(x)$, then the cone angle of the corresponding conical toric K\"ahler metric  is $2\pi \beta_j$ along the toric divisor determined by $l_j(x)=0$.
\end{proposition}

Let $\omega = \ddbar \phi$ be a smooth $\beta$-weighted K\"ahler metric and let $u=\mathcal{L} \phi$. Then $u(x)= \sum_j \beta_j^{-1} l_j(x) \log l_j(x) + f(x)$ for some $f\in C^\infty(\overline P)$.

\begin{example} Let $P=[0, 1]$ then the associated toric manifold is $X=\mathbb{P}^1$ with the polarization $\mathcal{O}(1)$. We consider the  symplectic potential by $$u= (\beta_1)^{-1} x \log x + \beta_2^{-1} (1-x) \log (1-x).$$
Then
$$\rho= \log |z|^2 = u'(x) = (\beta_1^{-1}- \beta_2^{-1} ) + \log \frac{ x^{\beta_1^{-1}}}{(1-x)^{\beta_2^{-1}}}, ~\text{or } ~ |z|^2 = \frac{ x^{\beta_1^{-1}}}{(1-x)^{\beta_2^{-1}}} e^{\beta_1^{-1} - \beta_2^{-1}} $$
and so
$$x \sim |z|^{2\beta_1} \text{ near } 0, ~~~ (1-x) \sim |z|^{-2\beta_2} \text{ near } \infty . $$
In particular, $x$ is a smooth function in $|z|^{2\beta_1}$ near $z=0$ and $(1-x)$ is a smooth function in $ |z|^{-2\beta_2}$ near $z=\infty$.  The K\"ahler potential $\varphi$ is given by
$$\varphi (\rho) = x (\beta_1^{-1}- \beta_2^{-1} )  - \beta_2^{-1}  \log (1-x) . $$
Hence $\ddbar \varphi$ extends to a conical metric with cone angle $2\pi \beta_1$ and $2\pi\beta_2$ at $[z=0]$ and $[z=\infty]$ respectively.
When $\beta= \beta_1 = \beta_2$,  $\varphi= \beta^{-1} \log (1+ |z|^{2\beta}) $ and the $\omega = 2 \ddbar \varphi$ is a smooth conical K\"ahler-Einstein metric in $c_1(\mathbb{P}^1)$ satisfying $$\Ricc(\omega) = \beta \omega + (1-\beta) ([z=0] + [z=\infty]). $$

\end{example}

\begin{lemma} Let $g$ be a smooth $\beta$-weighted toric K\"ahler metric  on a toric manifold $X_P$.  Let $D$ be the toric divisor such that $g$ is a smooth toric K\"ahler metric on $X\setminus D$. Then for any $k\geq 0$, there exists $C_k>0$ such that for any point $p \in X\setminus D$,
\begin{equation}
|\nabla_g^k Rm(g)|_{g} (p) \leq C_k .
\end{equation}

\end{lemma}

\begin{proof} The calculation of $|\nabla_g^k Rm(g)|_{g} (p)$ can be locally carried out on the $\beta$-covering space for each coordinate chart $(p, \{v_i\}_{i=1}^n)$ as in the orbifold case. All the quantities must be  bounded because  the $g$ is a smooth toric K\"ahler metric after being lifted on the covering space.

\end{proof}

We now can solve a Monge-Amp\`ere equation with smooth $\beta$-weighted data.

\begin{proposition} \label{tormaso}

Let $\omega$ be a smooth $\beta$-weighted toric K\"ahler metric on a toric manifold $X_P$. Then for any smooth $\beta$-weighted function $f$ on $X_P$ with $\int_{X_P} e^{-f} \omega^n = \int_{X_P} \omega^n$, there exists a unique $\beta$-weighted smooth function $\varphi$ satisfying
\begin{equation}\label{toricma1}
(\omega+\ddbar \varphi)^n = e^f \omega^n, ~\sup_{X_P} \varphi =0.
\end{equation}

\end{proposition}

\begin{proof}

We consider the following continuity method for $t\in [0, 1]$
\begin{equation}\label{toriccon1}
(\omega+ \ddbar \varphi_t)^n = e^{t f + c_t} \omega^n,
\end{equation}
where $c_t$ is determined by $\int_{X_P} e^{tf + c_t} \omega^n = \int_{X_P} \omega^n.$
Let
$$S=\{ t \in [0, 1]~|~ \text{ (\ref{toriccon1}) is solvable for } t \text{ with } \varphi_t \in C^{\infty}_\beta (X_P) \}.$$
Obviously, $0\in S$. $S$ is open by applying the implicit function theorem for the linearized operator
$$\Delta_{\beta, t} : C^{k+2, \alpha}_\beta(X_P) \rightarrow C^{k, \alpha}_\beta (X_P).$$
All the local calculation can be carried out in the $\beta$-covering space on each coordinate chart $(p, \{v_i\}_{i=1}^n)$ because all the data involved are invariant under $(S^1)^n$-action. It suffices to prove uniform a priori estimates for $\varphi_t$ in $C^{k}_\beta(X_P)$ for $t\in [0, 1]$.

\noindent{\it $C^0$-estimates.} Let $\Omega$ be a smooth volume form on $X_P$. Then $\frac{e^{tf + c_t} \omega^n}{\Omega}$ lies in $L^{1+\epsilon}(X_P, \Omega)$ for some $\epsilon>0$. By Yau's Moser iteration \cite{Y1} adapted to the conical case or by Kolodziej's $L^\infty$-estimate \cite{Ko}, there exists $C>0$ such that for all $t\in [0, 1]$, if $\varphi_t\in C^\infty_\beta(X)$ solves (\ref{toriccon1}),
$$||\varphi_t||_{L^\infty(X_P)} \leq C.$$

\noindent{\it Second order estimates.}    We consider
$$H_t  = \log tr_{\omega} (\omega_t) - A\varphi_t.$$
Suppose at $t \in [0, 1]$, $\sup_{X_P} H_t = H_t (q)$. We lift all the calculation on the $(S^1)^n$-invariant $\beta$-covering space in a fixed local coordinate chart $w\in \mathbb{C}^n$.
Standard calculations show that near $\tilde q$, there exists $C>0$ such that
\begin{eqnarray*}
\tilde{\Delta}_{t, \beta} \tilde{H}_t &\geq& - Ctr_{\tilde \omega_t}(\tilde \omega) - An + A tr_{\tilde \omega_t}(\tilde \omega)\\
&\geq& \frac{A}{2} tr_{\tilde \omega_t} (\tilde \omega) - C,\\
\end{eqnarray*}
where $\tilde q$, $\tilde \omega$ and $\tilde \omega_t$ are the lifting of $q$, $\omega$ and $\omega_t$.
By the maximum principle, at $\tilde q$, $tr_{\tilde \omega_t} (\tilde \omega) $ is bounded above by a constant independent of $t\in [0, 1]$. Combining the equation (\ref{toricma1}), $tr_{\tilde \omega}(\tilde \omega_t)$ is also bounded above by a constant independent of $t$. Hence there exists $C>0$ such that for for all $t\in [0, 1]$, if $\varphi_t\in C^\infty_\beta(X)$ solves (\ref{toriccon1}),
$$ C^{-1} \omega \leq \omega_t \leq C \omega.$$

\noindent{\it Higher order estimates. } Calabi's third order estimates can be applied the way as in \cite{Y1, PSS} by the maximum principle after lifting all the local calculations on the $(S^1)^n$-invariant covering space. The Schauder estimates can also be applied by the bootstrap argument. Eventually, for any $k>0$, there exists $C_k$ such that for all $t\in  [0, 1]$, if $\varphi_t\in C^\infty_\beta(X)$ solves (\ref{toriccon1}),
$$ ||\varphi_t||_{C^k_\beta(X_P)} \leq C_k.$$

\end{proof}

%%%%%%%%%%%%%%%%%%%%%%%%%%%%%%%%%%%

\subsection{Proof of Theorem \ref{main3}}

An $n$-dimensional integral  polytope $P$ is called a Fano polytope if it is a Delzant polytope and $\lambda_i=1$ for each defining function $l_i (x) = v_i \cdot x + \lambda_i$.  Immediately, we have $0 \in P$. The toric manifold $X_P$ associated to $P$ is a Fano manifold. Each $(n-1)$-face of $P$ corresponds to a toric divisor of $P$. Then the union $D_P= \sum_j D_j$ for all the boundary divisors lies in $c_1(X)=c_1(-K_X)$, where $D_j$ is the toric divisor induced by the face $\{ l_j(x)=0\}$. In particular, $[D]$ is very ample.

The Futaki invariant of $X_P$ with respect to $(S^1)^n$-action is shown in \cite{M} exactly the Barycenter of $P$ defined by
$$P_c = \frac{ \int_P x dV}{\int_P dV},$$
where $dV=dx_1 dx_2 ... d x_n$ is the standard Euclidean volume form.

The following theorem is due to Wang and Zhu \cite{WZ} for the existence of K\"ahler-Einstein metrics on toric Fano manifolds.

\begin{theorem} There exists a smooth toric K\"ahler-Einstein metric on a toric Fano manifold $X_P$ if and only if the Barycenter of $P$ coincides with $0$.

\end{theorem}

If the Barycenter is not at the origin, it is also proved in \cite{WZ} that there exists a toric K\"ahler-Ricci soliton on $X_P$. The following theorem is proved by Li \cite{LiC} to calculate the the greatest Ricci lower bound $R(X)$.

\begin{theorem} Let $X_P$ be a toric Fano manifold associated to a Fano Delzant polytop $P$. Let $P_c$ be the Barycenter of $P$ and $Q\in \partial P$ such that the origin $O \in \overline{P_cQ}$.  Then the greatest Ricci lower bound of $X_P$ is given by
\begin{equation}
R(X_P) = \frac{ |\overline{OQ}| }{ |\overline{P_c Q}| }.
\end{equation}

\end{theorem}

%Let $\hat u(x) = \sum_j \beta_j^{-1} l_j(x) \log l_j(x)$ be a symplectic potential on $P$.  Let $\hat \phi (\rho)= \mathcal{L} \hat u (\rho)$ be the associated toric K\"ahler-potential. Then $\hat\omega = \ddbar \hat \phi$ extends to a smooth conical toric K\"ahler metric. If $D_j$ is the toric divisor defined by the face of $l_j(x)=0$. Then the cone angle of $\hat\omega$ along $D_j$ is exactly $2\pi \beta$.

For any $\tau \in \mathbb{R}^n $,  we define the divisor $D(\tau)$ by
\begin{equation}
D(\tau) = \sum_{j=1}^N l_j(\tau) D_j.
\end{equation}
$D(\tau)$ is a  Cartier $\mathbb{R}$-divisor in $c_1(X)$ and it is effective if and only if $\tau\in \overline{P}$. The defining section $s_\tau$ of $D(\tau)$ is given by the monomial
$$s_\tau = z^\tau.$$
Although $s_\tau$ is only locally defined,
$$|s_\tau|^2=|z|^{2\tau}= e^{\tau\cdot \rho} $$
 is globally defined and $|s_\tau|^{-2}$ induces a singular hermitian metric on $-K_X$ and can be viewed as a singular volume form with poles along $D_j$ of order $l_j(\tau)$. We immediately have the following lemma.

 \begin{lemma}\label{divisord}  If $R(X) <1$, then the $\mathbb{R}$-divisor $D(\tau)$ with $\tau = -\frac{\alpha}{1-\alpha} P_c$ is effective  if and only if $\alpha \in [0, R(X)]$.

 \end{lemma}

We consider the following real Monge-Amp\`ere equation on $\mathbb{R}^n$ for a convex function $\phi$
\begin{equation}\label{make}
\det (\nabla^2 \phi) = e^{ - \alpha \phi - (1-\alpha) \tau \cdot \rho}.
\end{equation}
Let $u = \mathcal{L} \phi$ be the symplectic potential. Then $$\det(\nabla^2 u)= \left(  \det (\nabla^2 \phi)^n\right)^{-1}$$  and  dual Monge-Amp\`ere equation for $u$  is given by
\begin{equation}
\det(\nabla^2 u) = e^{-\alpha u + (  \alpha x + (1-\alpha) \tau ) \cdot \nabla u}.
\end{equation}
%
%Let $s_\tau$ be the If we let $\varphi = \phi - \phi_0$ and let $\Omega_{\hat\omega}$ be a conical volume form such that $\ddbar\log \Omega_{\hat\omega} = - \hat\omega$, then equation (\ref{make}) becomes \begin{equation}
%(\hat\omega + \ddbar\varphi)^n = \left( e^{- \varphi} \Omega_{\hat\omega} \right)^{\alpha} |s_\tau|^{-2(1-\alpha)}\end{equation}
%
%
If we let $\omega= \ddbar \phi$, the corresponding curvature equation is given by
$$\Ricc(\omega) = \alpha \omega + (1-\alpha) [D(\tau)].$$

\begin{lemma} \label{torickefr}Suppose there exists a smooth conical K\"ahler-Einstein metric $\omega= \ddbar \phi$ satisfying $$\Ricc(\omega) = \alpha \omega + (1-\alpha) [D(\tau)]$$ for some $\alpha \in (0, R(X)]$ and $\tau \in \overline{P}$. Then
\begin{equation}
\tau =-  \frac{ \alpha}{1-\alpha} P_c, ~for~\alpha\neq 1,~\text{ }~ \tau =0, ~for ~\alpha=1.
\end{equation}
Furthermore, there exists $f\in C^\infty(\overline{P})$ such that
\begin{equation}
\mathcal{L}\phi(x)= \sum_{j=1}^N \beta_j^{-1} l_j (x) \log l_j(x) + f(x), ~~ \beta_j = \frac{ l_j (P_c)}{ l_j(0)} \alpha.
\end{equation}

\end{lemma}

\begin{proof} Consider the corresponding Monge-Amp\`ere equation $(\det \nabla^2 \phi)^n = e^{-\alpha \phi - (1-\alpha) \tau \cdot \rho}$.  The right hand side $e^{-\alpha \phi - (1-\alpha) \tau \cdot \rho}$ is integrable on $\mathbb{R}^n$ and in fact
$$\int_{\mathbb{R}^n} e^{-\alpha \phi - (1-\alpha) \tau\cdot \rho} d \rho = \int_{\mathbb{R}^n} \det (\nabla^2 \phi) d \rho = \int_X \omega^n = c_1(X)^n.$$
Then the Monge-Amp\`ere mass $\det(\nabla^2\phi) d\rho$ becomes $dx$ by the moment map and
\begin{eqnarray*}
0&=& \int_{\mathbb{R}^n} \nabla \left( e^{-\alpha \phi - (1-\alpha) \tau \cdot \rho} \right) d \rho\\
&=& - \int_{\mathbb{R}^n} (\alpha \nabla \phi + (1-\alpha) \tau) e^{-\alpha \phi - (1-\alpha) \tau \cdot \rho} d \rho\\
&=&  - \int_P (\alpha x + (1-\alpha) \tau) d x\\
&=& - (\alpha P_c + (1-\alpha) \tau) \int_P dx.
\end{eqnarray*}
Therefore $\tau = \frac{\alpha}{1-\alpha} P_c$ for $\alpha\neq 1$.

Suppose $u (x) = \mathcal{L}\phi(x) = \sum_{j=1}^N \beta_j^{-1} l_j(x) \log l_j(x) + f(x).$ The Monge-Amp\`ere equations for $\phi$ and $u$ are given by
$$\det(\nabla^2 \phi) = e^{ -\alpha ( \phi + P_c \cdot \rho)}, ~~ \det(\nabla^2 u ) = e^{-\alpha ( u - (x-P_c) \cdot \nabla u) }. $$

Without loss of generality, we assume that  $l_1(x)=l_2(x)= ...= l_n(x)=0$  with $l_i(x)=v_i\cdot x+1,\ 1\leq i\leq n$ defines a vertex $p$ of $P$. Then there exists a smooth positive function $F(x)$ on any compact subset $U$ of $\overline{P}$ with $U\cap \{ l_j(x)=0 \} =\phi$ for all $j>n$,  such that
$$\det ( \nabla^2 u) = \frac{F(x)}{l_1(x) l_2(x)... l_n(x)}.$$
On the other hand,
$$
u(x) - (x -P_c) \nabla u(x) =\sum_{j=1}^N \beta_j^{-1} l_j( P_c) \log l_j(x) - \sum_{j=1}^N  \beta_j^{-1} (x- P_c) \cdot v_j - (x-P_c)\cdot\nabla f(x)
$$
and so
$$e^{-\alpha( u - (x-P_c)\cdot \nabla u)} = \frac{ e^{ - \alpha (x-P_c) \cdot ( \nabla f(x) + \sum_j \beta_j^{-1} v_j)}  }{ \prod_{j=1}^N \left( l_j(x)^{\alpha \beta_j^{-1} l_j(P_c)} \right)}.
$$
By comparing the power of $l_j(x)$, we have
$$ \beta_j = l_j (P_c) \alpha = \frac{l_j(P_c)}{l_j(0)} \alpha$$
since $l_j(0)=1$. This completes the proof of the lemma.

\end{proof}

Lemma \ref{torickefr} tells us that we should consider the following Monge-Amp\`ere equation
\begin{equation}\label{toricidma}
\det (\nabla^2 \phi)^n = e^{-\alpha (\phi - P_c \cdot \rho)}.
\end{equation}
The right hand side of  equation (\ref{toricidma}) is always integrable because $P_c$ lies in $P$ and $\phi - \log (\sum_k e^{p_k \cdot \rho })$ is bounded on $\mathbb{R}^n$ where $p_k$ runs over all vertices of $P$.

For each $\alpha\in (0, 1]$, we define $\beta=\beta(\alpha)= (\beta_1, \beta_2, ..., \beta_N)$ by $\beta_j= \frac{l_j(P_c)}{l_j(0)}\alpha$.
%

%\begin{lemma} Let $\overline{\omega} \in c_1(X) $ be a smooth toric K\"ahler metric on $X_P$. If  $\omega\in c_1(X)$ is a $C^\infty_{\beta(\alpha)}$ K\"ahler metric for some $\alpha\in (0, 1]$ , then there exists $C>0$ such that
%
%$$ C^{-1}     e^{\alpha \left( \log \left(\sum_k e^{p^{(k)} \cdot \rho}\right)    - P_c \cdot \rho \right)}  \leq \frac{ \omega^n}{\overline{\omega}^n} \leq C e^{\alpha \left( \log \left(\sum_k e^{p^{(k)} \cdot \rho}\right)    - P_c \cdot \rho \right)},$$
%
%where $p^{(k)}$ runs over the vertices of $P$.

%\end{lemma}

%\begin{proof} Let $s$ be the defining section of the divisor $D(P_c)$. $(\sum_k e^{p^{(k)}\cdot \rho})^{-1}$ can be viewed as a smooth hermitian metric $h$ on $-K_X$. Then
%
%$$|s|_h^{-2} = e^{\alpha \left( \log \left(\sum_k e^{p^{(k)} \cdot \rho} \right)    - P_c \cdot \rho\right )}.$$
%
%On the other hand, $$\frac{\omega^n}{|s|_h^2 \overline{\omega}^n}$$ is uniformly bounded above and below from $0$. The lemma then follows.

%\end{proof}

\begin{lemma} \label{inikm}For any $\alpha\in (0, 1]$, there exists a $C^\infty_{\beta(\alpha)}$  conical toric K\"ahler metric $\omega$ such that
$$\Ricc(\omega) = \alpha \theta + (1-\alpha) [ D(\tau)],$$
where $\theta \in c_1(X)$ is a fixed $C^\infty_{\beta(\alpha)}$ toric K\"ahler metric and  $\tau= \frac{\alpha}{1-\alpha} P_c$. In particular, $\Ricc(\omega) >0$, if $\alpha \in (0, R(X))$.

\end{lemma}

\begin{proof}

It suffices to prove for $\alpha\in (0,1)$.  Let $\hat u (x) = \sum_j \beta_j^{-1} l_j(x) \log l_j(x)$ for $\beta_j = \frac{l_j(P_c)}{l_j(0)} \alpha$. Let $\hat\phi = \mathcal{L}\hat u$ and $\hat\omega = \ddbar \hat \phi$.  Then there exists a $C^\infty_{\beta(\alpha)}$ real valued $(1,1)$-form $\eta \in c_1(X)$ and a divisor $D(\tau)$ with $\tau = -\frac{\alpha}{1-\alpha} P_c$ such that
$$\ddbar \log \hat\omega^n = \alpha \eta + (1-\alpha) [D(\tau)].$$
This is because along each $D_j $ defined by $l_j(x)=0$, $\hat\omega^n$ has a pole of order $$1-\beta_j=1-l_j(P_c)\alpha$$
and
$$(1-\alpha)D(\tau)=(1-\alpha) \left(\sum_j l_j \left(\frac{\alpha}{1-\alpha} P_c\right) D_j \right)= \sum_j (1- l_j(P_c) \alpha) D_j. $$

Since $\eta\in c_1(X)$ is $C^\infty_{\beta(\alpha)}$, there exists $\psi\in C^\infty_{\beta(\alpha)}$ such that $\eta+ \ddbar \psi$ is a $C^\infty_{\beta(\alpha)}$ toric K\"ahler metric.
We can assume that $\int_{X_P} e^\psi \hat\omega^n = \int_{X_P} \hat\omega^n$ after a constant translation.

Then by Proposition \ref{tormaso}, the following equation
$$(\hat\omega + \ddbar \varphi)^n = e^{\psi} \hat\omega^n, ~~\sup_{X_P} \varphi =0$$
admits a unique $C^\infty_{\beta(\alpha)}$ solution $\varphi$. By letting  $\omega = \hat\omega + \ddbar \varphi$ and $\theta = \eta+ \ddbar \psi$, we have
$$\Ricc(\omega) = \alpha\theta + (1-\alpha) [D_(\tau)].$$

\end{proof}

By Lemma \ref{inikm}, we can choose a $C^\infty_{\beta(\alpha)}$ K\"ahler potential $\phi_0$ with $\omega_0 = \ddbar \phi_0$ such that
$$ \Ricc(\omega_0) = \alpha \theta + (1-\alpha)D(\tau)$$ for a $C^\infty_{\beta(\alpha)}(X_P)$ K\"ahler metric $\theta \in c_1(X_P)$ and $\tau= \frac{\alpha}{1-\alpha} P_c$ if $\alpha\neq 0$.
Let $$w=\frac{1}{\alpha} \left(  - \alpha P_c \cdot \rho -  \log \det(\nabla^2  \phi_0) \right).$$
Then
$$\ddbar w =\theta $$
and $ |w - \hat\phi|$ is uniformly bounded by the argument in  Lemma \ref{inikm}. This implies that  $w$ is a $C^\infty_{\beta(\alpha)}$ K\"ahler potential  and we have
$$\det(\nabla^2 \phi_0 ) = e^{-\alpha (w -  P_c \cdot \rho) }.$$

We will then define the following continuity method for a family of Monge-Amp\`ere equations for $t\in [0, \alpha]$
\begin{equation}\label{toriccon2}
\det(\nabla^2 \phi_t) = e^{ - t (\phi_t - P_c \cdot \rho) - (\alpha -t) w }.
\end{equation}
Let $\varphi_t = \phi_t - \phi_0$ and define $C^\infty_{\beta(\alpha)}$  functions $h_{\omega_0}$ and $h_{\theta}$  on $X_P$ by
$$ - \ddbar \log \omega_0^n - \ddbar h_{\omega_0} = \alpha \omega_0 + (1-\alpha)[D(\tau)]$$
and
$$
- \ddbar \log \theta^n - \ddbar h_{\theta} = \alpha \theta + (1-\alpha)[D(\tau)], ~ \int_{X_P} e^{h_\theta}\theta^n = \int_{X_P} \theta^n.
$$
Then equation (\ref{toriccon2}) is equivalent to
\begin{equation}
(\ddbar \omega_0 + \ddbar \varphi_t)^n = e^{- t\varphi_t } \left( e^{h_{\omega_0}} \omega_0 \right)^{\frac{t}{\alpha}} \left(e^{h_\theta} \theta \right)^{\frac{\alpha-t}{\alpha}},
\end{equation}
where $h_{\omega_0}$ and $h_{\theta}$ are $C^\infty_{\beta(\alpha)}$ on $X_P$ with
$$ - \ddbar \log \omega_0^n - \ddbar h_{\omega_0} = \alpha \omega_0 + (1-\alpha)[D(\tau)], ~~ - \ddbar \log \theta^n - \ddbar h_{\theta_0} = \alpha \theta + (1-\alpha)[D(\tau)]$$
Let
$$S=\{ t \in [0, \alpha]~|~ \text{ (\ref{toriccon2}) is solvable for } t \text{ with } \ddbar \phi_t \in C^{\infty}_\beta (X_P) \}.$$
Obviously, $\phi_0$ solves (\ref{toriccon2}) for $t=0$ and so $S\neq \phi$.
Notice that
$$\Ricc(\omega_t) = t\omega_t + (\alpha -t) \theta + (1-\alpha) [D(\tau)] \geq t \omega_t$$
for $t\in [0, \alpha]$ and $\tau = \frac{\alpha}{1-\alpha} P_c$ if $\alpha\neq 1$.  It implies that the first eigenvalue of the Laplace operator $\Delta_t = \tr_{\omega_t} (\ddbar )$ is strictly greater than $t$. By the argument in Proposition \ref{tormaso}, $S$ is open and it suffices to show that $S$ is closed by proving uniform a priori estimates for $\phi_t - \phi_0$.

\begin{proposition}

There exists $C>0$ such that for all $t\in [0, \alpha]$,
\begin{equation}
||\phi_t - \phi_0||_{L^\infty(\mathbb{R}^n)} \leq C.
\end{equation}

\end{proposition}

\begin{proof} We fix some positive $\epsilon_0 \in S$.  We let 
$$\Phi_t =\alpha^{-1} \left( \phi_t - P_c\cdot \rho \right) \text{ and } W= \alpha^{-1} ( w - P_c\cdot \rho).$$ Then the equation (\ref{toriccon2}) becomes
$$\det \nabla^2 \Phi_t = e^{- (\Phi_t + W)}.$$ Let $W_t= \Phi_t + W$. Immediately, we can see that the moment map with respect to $\Phi_t$ is given by  $$F_t: \nabla \Phi_t \rightarrow P-P_c$$ whose image is the translation of $P$ by $-P_c$. In particular, the Barycenter of the new polytope $P- P_c$  coincides with the origin.

Suppose $$m_t = W_t(\rho_t) = \inf_{\mathbb{R}^n} W_t (\rho) $$ for a unique $\rho_t\in \mathbb{R}^n$ since $W_t$ is asymopotically equivalent to $ \log (\sum e^{( p_k - P_c) \cdot  \rho} )$ where $p_k$ are the vertices of $P$. We can apply the same argument by Wang-Zhu in \cite{WZ}. First one can show by John's lemma and the maximum principle (see Lemma 3.1, 3.2 in \cite{WZ}), that there exists $C>0$ such that for all $t\in [\epsilon_0, \alpha]$,
$$m_t = W_t(\rho_t) = \inf_{\mathbb{R}^n} W_t (\rho) \leq C.$$
Then by making use of the fact that the Barycenter of $P-P_c$ lies at the origin $O$, the same argument in \cite{WZ} (Lemma 3.3) shows that there exists $C>0$ such that for all $t\in [\epsilon_0, \alpha]$,
$$|\rho_t| \leq C.$$
This then implies that $$\varphi_t = \alpha^{-1}(\Phi_t - W)$$ is uniformly bounded above for $t\in [\epsilon_0, \alpha]$ by the same argument in Lemma 3.4 in \cite{WZ}.

The uniform lower bound of $\varphi_t$ can be obtained  either  by the Harnack inequality
$$-\inf_{X_P} \varphi_t \leq C(1+\sup_{X_P} \varphi_t)$$
 adapted from the smooth case or directly by the argument in Lemma 3.5 in \cite{WZ}.

\end{proof}

\begin{lemma} \label{torest2}

For any $k\geq 0$, there exists $C_k>0$ such that for all $t\in [0, \alpha]$,
\begin{equation}
||\varphi_t||_{C^k(X_P)} \leq C_k.
\end{equation}

\end{lemma}

\begin{proof} The Laplace $\Delta_{\beta(\alpha), t} \varphi_t$ is uniformly bounded by Yau's estimates after lifting the calculations to the $\beta(\alpha)$-covering space as in the proof of Proposition \ref{tormaso}. The $C^3$-estimates and the Schauder estimates can be applied in the same way.

\end{proof}

\begin{theorem}\label{toke} Let $X_P$ be a toric Fano manifold.

\medskip

\begin{enumerate}

\item For any $\beta \in (0, R(X_P))$, there exist a unique smooth toric conical K\"ahler-Einstein metric $\omega$ and a unique effective toric $\mathbb{R}$-divisor $D_\beta \in |-K_X|$ satisfying

$$\Ricc(\omega) = \beta \omega + (1-\beta)[D_\beta].$$

\medskip

\item For $\beta = R(X_P)$, there exists a unique smooth toric conical K\"ahler-Einstein metric $\omega$ satisfying
\begin{equation}\label{torkeeqn}
\Ricc(\omega ) = R(X_P) \omega + (1-R(X_P)) [D_P]
\end{equation}
for an effective $\mathbb{Q}$-divisor $D_P$ in $c_1(X)$.
In particular, if $D_j$ is the toric divisor associated to the face defined by $l_j(x) =v_j \cdot x  + \lambda_j=0$, then $$D_P = \sum_j \frac{1-\beta_j}{1-R(X)}  D_j, ~~\beta_j = \frac{ l_j (P_c)} { l_j(0) } R(X_P)$$ and the cone angle of $\omega$ along $D_j$ is $2\pi \beta_i$,  if $R(X)<1$.

\medskip

\item For $\beta \in (R(X), 1]$, there does not exist a smooth toric conical K\"ahler-Einstein metric $\omega$ satisfying
$$\Ricc(\omega) = \beta \omega  + (1-\beta) [D],$$
with an effective $\mathbb{R}$-divisor $D_\beta$ in $[-K_X]$.

\end{enumerate}

\end{theorem}

\begin{proof} (1) and (2) are proved by the uniform estimates from Lemma \ref{torest2}.  If $\beta>R(X_P)$, there still exists a smooth conical K\"ahler-Einstein metric satisfying $Ric(\omega) = \beta \omega+ (1-\beta)[D_\beta]$ for some toric divisor $D$, however, by Lemma \ref{divisord}, $D$ is not effective and so (3) is proved.

\end{proof}

\begin{corollary} \label{torlowbd} Let $X$ be a Fano toric manifold and $\omega\in c_1(X)$ be a smooth K\"ahler metric. We define $$F_{\omega, \alpha}(\varphi) = J_\omega(\varphi) -\frac{1}{V}\int_X\varphi \omega^n  - \frac{1}{\alpha} \log \frac{1}{V} \int_X e^{-\alpha \varphi} \omega^n$$ for all $\varphi\in C^\infty(X)\cap PSH(X, \omega)$. Suppose $R(X)<1$. Then

\begin{enumerate}

\item for $\alpha\in (0,R(X))$, $F_{\omega, \alpha}$ is  $J$-proper.

\item for $\alpha = R(X)$, $F_{\omega, \alpha}$ is bounded below.

\item for $\alpha\in (R(X), 1]$, $\inf_{\varphi \in PSH(X, \omega)\cap C^\infty(X)} F_{\omega, \alpha}(\varphi) = - \infty.$

\end{enumerate}

\end{corollary}

\begin{proof} It suffices to prove $(2)$ by Corollary \ref{flowbd} . This can be proved by modifying the argument in \cite{BBGZ, Be}. By Theorem \ref{toke}, there exists a unique $(S^1)^n$-invariant $\psi \in L^\infty(X_P) \cap PSH(X_P, \omega)$ satisfying
$$(\omega+ \ddbar \psi) = e^{-\alpha \psi}\mu, ~\text{ }~\mu = \left(\Omega_\omega\right)^\alpha (\Omega_D)^{1-\alpha},$$
where $\Omega_\omega$ is a smooth volume form with $\ddbar\log\Omega_\omega= -\omega$ and $\Omega_D$ is a positive $(n,n)$-current with $-\ddbar\log \Omega_D=[D]$ and $D=D\left(\frac{\alpha}{1-\alpha} P_c \right)$.
For any $\varphi\in PSH(X, \omega)\cap L^\infty(X)$, let $\varphi_t$ be the weak geodesic $\varphi_t$ joining $\psi$ and $\varphi$ with $\varphi_0=\psi$ and $\varphi_1=\varphi$. Then the modified functional
$$f(t)=\FF_{\omega, \alpha}(\varphi_t) = J_\omega(\varphi_t) - \frac{1}{V}\int_{X_P} \varphi_t \omega^n - \frac{1}{\alpha} \log \frac{1}{V}\int_{X_P} e^{-\alpha \varphi}\mu$$ is convex on $[0, 1]$ and $f'(0)\geq 0$ by applying the same argument in Theorem 6.2 in \cite{BBGZ}. This shows that $\FF_{\omega, \alpha}$ is bounded below and since $\Omega_D$ is bounded below away from $0$,  and therefore $F_{\omega, \alpha}$ is bounded from below as well.

\end{proof}

\begin{example} Let $X$ be $\mathbb{P}^2$ blown-up at one point. Then $R(X) = 6/7$ shown in \cite{Sze} and $X$ admits a holomorphic $\PP^1$ fibre bundle $\pi: X \rightarrow \PP^1$. Let $D_\infty$ be the  infinity section of $\pi$ and  $H_1$ and $H_2$ be the two toric $P^1$ fibre of $\pi $. Then the divisor $D_P$ in the equation (\ref{torkeeqn}) is given by $$D_P= 2D_\infty + (H_1+H_2)/2. $$  

\end{example}

We remark that  for $\beta\in (R(X), 1]$, there exists a smooth toric conical K\"ahler-Einstein metric in $c_1(X_P)$ satisfying $\Ricc(g) = \beta g  + (1-\beta) [D_\beta]$ for some toric divisor $D_\beta$ in $c_1(X_P)$ which can not  be effective. However, it is not a conical K\"ahler metric we are interested in, because   the Ricci current of such K\"ahler-Einstein metric is not bounded below. We also notice that in the case of $\beta = R(X)<1$, the conical K\"ahler-Einstein metric is not necessarily invariant under the maximal compact subgroup of the automorphism group of $X$ by the following example.

\begin{example} Let $X$ be $\mathbb{P}^2$ blown-up at one point. It is a compactification of $\mathbb{C}^2\setminus \{0\}$ by adding one divisor $D_0 =\mathbb{P}^{1}$ at $0$ and another divisor $D_\infty= \mathbb{P}^1$ at $\infty$. The maximum compact group of the automorphism group of $X$ is $U(2)$. The only divisor in $c_1(X)$ fixed by the $U(2)$-action is given by
$$D= 3[D_\infty]- [D_0].$$ Therefore, there does not exists a smooth conical K\"ahler-Einstein metric on $X$ with positive Ricci current. In fact, there does not exist any conical K\"ahler-Einstein metric for such $D$ by the following observation, although one can look for complete conical K\"ahler-Einstein metric in this case.   

Let $z=(z_1,  z_2) \in \mathbb{C}^2$ and $\rho= \log |z|^2=\log |z_1|^2 + |z_2|^2 \in \mathbb{R}$. Then a convex function $u=u(\rho)$ satisfies the Calabi symmetry if
\begin{enumerate}
\item $u''>0$,
\item there exists $0<a<b$ and $\beta_1, \beta_2>0$ such that $U_0'(0)>0$, $U_\infty'(0)>0$, where $U_0(e^{\beta_1 \rho} )=u(\rho) - a\rho $ and $U_\infty(e^{-\beta_2 \rho}) = u(\rho) - b \rho$.

\end{enumerate}
Then $\ddbar u$ extends to a smooth conical K\"ahler metric $\omega_u$ on $X$. In particular, if we choose $b=3$ and $a=1$, $\omega_u \in c_1(X)$.
The conical K\"ahler-Einstein equation for $\omega_u$ is given by the following ODE
$$ u' u'' = e^{2\rho -\alpha u}.$$
Differentiating both sides after taking the log, we have
$$\frac{u'''}{u''} + \frac{u''}{u'} = 2 - \alpha u'.$$
Let $x= u'$ and  $Y(x) = u''(\rho)$. Then we have
$$Y'(x)+  (x^{-1} +\alpha) Y(x) =2, ~Y(1)=Y(3)=0$$
and so $$Y(x) =  \frac{2}{\alpha} - \frac{2}{\alpha^2 x} + \frac{c}{x e^{\alpha x}},   ~(3\alpha -1)e^{2\alpha} - \alpha +1=0.$$
Hence by the monotonicity of $(3\alpha -1)e^{2\alpha} - \alpha +1$, $\alpha =0$ but it contradicts the assumption $u'u''(\infty) = 0$.   On the other hand, there exist smooth conical K\"ahler-Einstein metrics in $b[D_\infty]- a[D_0]$ for some $0<a<b$ by solving the above ODE in the same way.

\end{example}

%%%%%%%%%%%%%%%%%%%%%%%%%%%%%%%%%%%%%%%%%%%

\bigskip

\section{The Chern number inequality}

\subsection{Curvature estimates}\label{curv-est}

In this section, we will derive some curvature estimates for a smooth conical K\"ahler metric whose Ricci curvature is bounded.

Let $D$ be the smooth simple divisor of $X$. At each point $p$ on $D$, we can use the following holomorphic local coordinates
$$z=(y, \xi) = (z_1, ..., z_{n-1}, z_n), ~y=(z_1, ..., z_{n-1}), ~\xi = z_n  $$
and $D$ is locally defined by $\xi=0$.  We write $\xi = r^{1/\beta} e^{i\theta}$ for $\theta\in [0, 2\pi)$.
We use Greek letters $\alpha, \beta, ... $ as indices for  $1, ..., n$ and letters $i, j, ...$ for $1, ..., n-1$.

Let us recall the following result by Jeffres, Mazzeo and Rubinstein in \cite{JMR} on a complete asymptotic expansion of smooth conical K\"ahler metrics. .

\begin{proposition}\label{expansion}
(\cite[Proposition 4.3]{JMR}) Let $\omega$ be a smooth conical  K\"ahler metric with conical singularity along a smooth simple divisor $D$ of angle $2\pi \beta$. Suppose $\varphi$ is a local potential of $\omega$, i.e., $\omega = \ddbar \varphi$ in a neighborhood of a conical point $(y, \xi)$, then the asymptotic exapansion of $\varphi$ takes the following form
\begin{equation}
\varphi(r, \theta, y)\sim \sum_{j, k, l\geq 0} a_{jkl}(\theta, z) r^{j+\frac{k}{\beta}} (\log r )^l.
\end{equation}
In particular,  if the Ricci curvature of $\omega$ is bounded and  $\beta \in (1/2, 1)$, $\varphi$ has the following expansion
\begin{equation}
\varphi(r, \theta, y) = a_{00}(y) + (a_{01}(y) \sin \theta + b_{01}(y) \cos\theta) r^{1/\beta} + a_{20} (y) r^2 + O(r^{2+\epsilon})
\end{equation}
for some $\epsilon(\beta)>0$.

\end{proposition}

When the Ricci curvature is bounded and $\beta\in (1/2, 1)$, $$\varphi = a(y) + b(y) (\xi + \bar \xi) + \sqrt{-1} c(y) (\xi -  \bar \xi) + d(y) |\xi|^{2\beta} + o (|\xi|^{2\beta + \epsilon}  )$$ for some $\epsilon>0$. From now on in this section, we will always assume that $g$ is a conical K\"ahler metric on $X$ with cone angle $2\pi \beta$ for $\beta \in (1/2, 1)$ along the simple smooth divisor $D$, in addition, the Ricci curvature of $g$ is bounded.

The following lemma can be obtained by straightforward calculations.

\begin{lemma}  \label{metor} Let $g$ be a smooth conical K\"ahler metric with cone angle $2\pi \beta $ along a smooth simple divisor $D$ for $\beta\in (0, 1)$.  Let $o(1)$ be the quantity satisfying $\lim_{|\xi|\rightarrow 0} o(1) =0$. Then
\begin{equation} g_{z_i \bar z_j}  \sim \delta_{i j}  + o(1) \end{equation}
\begin{equation} g_{\xi \bar \xi} \sim |\xi|^{-2(1-\beta)} + o(|\xi|^{-2(1-\beta)}). \end{equation}
\begin{equation} g_{z_i \bar \xi} \sim O(1). \end{equation}

\end{lemma}

By taking the inverse, we have the following corollary from Lemma \ref{metor}.

\begin{corollary}\label{asymp-g}  Let $g$ be a smooth conical K\"ahler metric with cone angle $2\pi \beta $ along a smooth simple divisor $D$  for $\beta\in (0, 1)$.

\begin{equation} g^{z_i \bar z_j} \sim   \delta_{ij}  + o(1) \end{equation}

\begin{equation} g^{\xi \bar \xi} \sim |\xi|^{2(1-\beta)} + o(|\xi|^{2(1-\beta)}). \end{equation}

\begin{equation} g^{ z_i \bar \xi} \sim |\xi|^{2(1-\beta)}. \end{equation}

\end{corollary}

The following lemmas give the estimates for the curvature tensor of $g$.

\begin{lemma}\label{metor2} Let $g$ be a smooth conical K\"ahler metric with cone angle $2\pi \beta $ along a smooth simple divisor $D$  for $\beta\in (0, 1)$. If the Ricci curvature of $g$ is bounded, then 
\begin{eqnarray} \label{cur1}  
R_{z_i \bar z_j z_k \bar z_l}  &\sim&  R_{z_i \bar z_j z_k \bar \xi} = O(1)\\
\label{cur2}  R_{z_i \bar z_j \xi  \bar \xi } &=& O (|\xi|^{-2(1-\beta)}).  \\
\label{cur3}  R_{\xi \bar z_j \xi  \bar z_l }  &=&  O(|\xi|^{-1}) . \\
 \label{cur4}   R_{\xi \bar \xi  \xi \bar z_l}  &=& O(|\xi|^{-1}) .\\
 \label{cur5}  R_{\xi \bar \xi  \xi  \bar \xi }  &=&  O(|\xi|^{-\max (1, 4(1-\beta)) }) . 
 \end{eqnarray}

\end{lemma}

\begin{proof} The estimates (\ref{cur1}), (\ref{cur2})  and (\ref{cur3}) can be shown by straightforward calculation using the curvature formula
$$R_{\alpha \bar \beta \gamma \bar \zeta} = - g_{\alpha\bar \beta,  \gamma \bar \zeta} + g^{\mu \bar \nu} g_{\alpha \bar \nu, \gamma} g_{\mu \bar \beta, \bar \zeta}. $$
The estimates (\ref{cur4}) and (\ref{cur5}) follow by combining the boundness of the Ricci curvature and the estimates (\ref{cur2}) and (\ref{cur3}).

\end{proof}

\begin{corollary} Let $g$ be a smooth conical K\"ahler metric with cone angle $2\pi \beta $ along a smooth simple divisor $D$  for $\beta\in [1/2, 1)$. If the Ricci curvature of $g$ is bounded, then 

$$ R^{z_i}_{\xi, z_k \bar z_l} \sim R^\xi_{\xi, z_k \bar z_l}= O(1).$$

$${R^{z_i}_{z_p, z_k \bar z_l} }\sim R^{\xi}_{z_p, z_k \bar z_l} = O(|\xi|^{2(1-\beta)}). $$

\end{corollary}

From the curvature estimates, we immediately have the following proposition. 
\begin{theorem}  Let $g$ be a smooth conical K\"ahler metric with cone angle $2\pi \beta $ along a smooth simple divisor $D$  for $\beta\in (0, 1)$. If the Ricci curvature of $g$ is bounded, then  We have the following pointwise estimates for $|R_m|^2$
$$| Rm(g) |^2_g = O( |\xi|^{-2 + 4(1-\beta)})  $$
Consequently,  the $L^2$- norm of $Rm(g)$ is bounded, i.e., there exists $C>0$ such that
\begin{equation}
\int_X | Rm(g) |^2_ g ~dVol(g) \leq C. 
\end{equation}

\end{theorem}

\medskip

\subsection{Chern forms for conical K\"ahler metrics}

Let  $g$ be a smooth conical K\"ahler metric with cone angle $2\pi \beta$ along a smooth simple divisor $D$. We let $\theta$ be the connection form on the tangent bundle $TX$ induced by $g$, so locally we may write $$ \theta^\alpha_\gamma = g^{\alpha\bar \beta} g_{\gamma \bar \beta, \eta} d z_\eta$$
and $\Omega^{\alpha}_\gamma = \dbar \theta^\alpha_\gamma$ be the curvature form of $\theta$. Then the total Chern class is defined by
$$ \det (tI + \Omega) = \sum_{i=0}^n t^{2(n-i)} c_i(\Omega) $$
and let $P_i(\Omega_1, ..., \Omega_n)$ being the polarization of $c_i(\Omega)=c_i(X, g)$ for the conical metric $g$.

Let $g_0$ be a smooth K\"ahler metric  and $\theta_0$ be the connection induced by $g_0$ as
$$(\theta_0)^\alpha_\gamma = (g_0)^{\alpha\bar \beta} (g_0)_{\gamma \bar \beta, \eta} d z_\eta.$$
Then  $\Omega_0= \sqrt{-1} \dbar \theta_0$ is the curvature form of $\theta_0$.

Let $\theta_t = t\theta + (1-t) \theta_0$ with curvature $$\Omega_t = \bar \partial \theta_t = t \Omega + (1- t) \Omega_0. $$
Then we have
$$c_2(X, \omega) - c_2(X, \omega_0) = 2 \sqrt{-1}  \int_0^1 \dbar P_2(\theta- \theta_0, \Omega_t) dt.$$

We will construct connections on the divisor $D$ from $\theta$ and $\theta_0$. Let $p$ be a point in the divisor $D$.  We can choose holomorphic local coordinates
$$z=(z_1, ..., z_n), ~ \xi = z_n$$ such that $D$ is locally defined by $\xi =0$
as in section \ref{curv-est}.
% where we keep the notations that $\alpha, \beta,... = 1, 2, ..., n$ and $i, j, ... = 1, 2, ..., n-1$.

\begin{definition}

We define $H$ and $H_0$ locally by 
\begin{equation} H_{i\bar j} = g_{i\bar j}, ~ H^{i\bar j} = (H^{-1})_ {i\bar j}, ~~H_{n\bar j} = g_{n \bar j},
\end{equation}
and
\begin{equation}
 (H_0)_{i\bar j} = (g_0)_{i\bar j}, ~ (H_0) ^{i\bar j} = (H_0^{-1})_{i\bar j}, ~~(H_0) _{n\bar j} = (g_0) _{n \bar j}.
 \end{equation}

\end{definition}

\begin{definition}
For each coordinate system $(z,\xi)$ chosen as above, we  define  $(1, 0)$- forms  $\theta_D$ and $\theta_{0, D}$ locally by
\begin{equation}
\theta_D = (\theta^i_i)|_D
\end{equation}
 and
\begin{equation}
\theta_{0, D} = (\theta_0)^i_i |_D + H^{i\bar j} H_{n\bar j} (\theta_0)^n_i |_D.
\end{equation}

\end{definition}

\begin{lemma}\label{c-th}
 The $(1, 0)$-form
\begin{equation}
\theta_D = (H^{i\bar j} H_{i \bar j, k}|_D)  d z_k = \partial \log \det H|_D
\end{equation}
defines  a global smooth Chern connection of the anti-canonical  line bundle of $D$.  In particular, its curvature form $\sqrt{-1}\dbar \theta_D$ is a smooth closed real $(1,1)$-form in $c_1(D)$.

\end{lemma}

\begin{proof}

By Corollary \ref{asymp-g}, we have $g^{i\bar n}|_D=0$ and hence
$$
(\theta_D)^i_i = g^{i\bar \beta} g_{i \bar\beta, k } dz_k  = (H^{i\bar j} H_{i\bar j, k}|_D )dz_k = \partial (\log \det H|_D) = \partial \log (\omega|_D)^{n-1}.$$
By Proposition \ref{expansion},  the regular part of $\omega$ restricted to  $D$ is a smooth K\"ahler form from the expansion in Proposition \ref{expansion} and for different holomorphic local coordinates $z=(z_1, ..., z_{n-1})$ and $w=(w_1, ..., w_{n-1})$ on $D$, $$\theta_D(z) = \theta_D(w) + \partial \log \left| \det \left( \frac{\partial z_i}{\partial w_j} \right) \right|^2.$$  Therefore $\theta_D $ defines a smooth connection on anti-canonical bundle of $D$.

\end{proof}

%%%%%%%%%%%%%%%%%%%%%%%%%

\begin{lemma} \label{c-th0}
 The $(1,0)$-form
\begin{eqnarray*}
\theta_{0, D}
&=& \left( (\theta_0)^i_i+(H^{i \bar j }H_{n \bar j} )(\theta_0)^n_i  \right) |_D\\
&=& \left( (g_0)^{i\bar \beta} (g_0)_{i \bar \beta, k} |_D dz_k + (H^{i \bar j }H_{n \bar j} )(g_0)^{n\bar \beta} (g_0)_{i \bar \beta, k}|_D dz_k\right)
\end{eqnarray*}
defines a smooth Chern connection of the anticanonical bundle  of $D$.  In particular, $\sqrt{-1} \dbar \theta_{0,D}$ is a smooth closed real $(1,1)$ form in $c_1(D)$.

\end{lemma}

\begin{proof}
 Since $g_0$ is smooth and the restriction of $H_{i\bar j}$, $H^{i\bar j}$, $H_{n, \bar j}$ to $D$ are all smooth by the asymptotic expansion in Proposition \ref{expansion}, $\theta_{0, D}$ is locally smooth. It suffices to show that they patch together give rise to a connection.

To do that we need to show that the transformation of $\theta_{0,D}$ under different coordinate charts satisfies the cocycle condition for the anticanonical bundle of $D$. Let $(z_1, ..., z_{n-1})$ and $(w_1, ..., w_{n-1})$ be two holomorphic local coordinates for some neighborhood of a point $p$ in $D$. Then they extend to two holomorphic coordinates $(z_1,.., z_{n-1},  z_n=\xi)$, $(w_1, ..., w_{n-1}, w_n=\eta)$ in a neighborhood of $p$ in $X$, where $D$ is locally defined by $\xi=0$ and $\eta=0$. Therefore for $i=1, ..., n-1$,
$$\left. \frac{\partial \eta}{\partial z_i} \right |_D= \left.\frac{\partial \xi}{\partial w_i} \right|_D = 0. $$
 By letting $$A^\alpha_\gamma= \frac{\partial z_\alpha}{\partial w_\gamma}, ~B^\alpha_\gamma = \frac{\partial w_\alpha}{\partial z_\gamma}, \tilde A^i_k= \frac{\partial z_i}{\partial w_k}, ~\tilde B^i_k = \frac{\partial w_i}{\partial z_k}, $$
we obtain along $D$,

$$A= \left( \begin{array}{cc}
    \tilde A & * \\
    0 & \frac{\partial \xi}{\partial \eta}  \\
  \end{array} \right),
~
B= \left( \begin{array}{cc}
    \tilde B & * \\
    0 & \frac{\partial \eta}{\partial \xi}  \\
  \end{array} \right),
~
A=B^{-1},
~\tilde A = \tilde B^{-1}.
$$
Straighforward computations show that

$$\theta_{0, D}(z) = \theta_{0, D}(w) + \partial \log |\det \tilde A|^2 $$
which completes the proof.

\end{proof}

For any smooth conical K\"ahler metric $\omega$ with cone angle $2\pi \beta$ along a smooth divisor $D$, we can define the first and second Chern classes $c_1(X, \omega)$ and $c_2(X, \omega)$. A priori,  the intersection numbers among $c_1(X, \omega)$ and $c_2(X, \omega)$ might depend on the choice of $\omega$ even if the Ricci curvature of $\omega$ is bounded. The following proposition relates the $c_1(X, \omega)$ and $c_2(X, \omega)$ to $c_1(X)$ and $c_2(X)$.

\begin{proposition}\label{an-char}
 Let $D$ be a smooth simple divisor on a K\"ahler manifold $D$. Suppose $\omega$ is a smooth conical K\"ahler metric with cone angle $2\pi \beta$ along $D$ with bounded Ricci curvature. Then
\begin{equation}\label{c1}
\int_X c_1(X, \omega) \wedge \omega^{n-1} = (c_1(X) -(1-\be)  [D]) \cdot [\omega]^{n-1} ,
\end{equation}
\begin{equation}\label{c2}
\int_X c_2(X, \omega) \wedge \omega^{n-2} = \left(c_2(X) +   (1-\beta)(-c_1(X) + [D]  \right) \cdot [D] ) \cdot [\omega]^{n-2},
\end{equation}
and
\begin{equation}\label{c12}
\int_X c_1^2(X, \omega) \wedge \omega^{n-2} = (c_1(X) -  (1-\be)[D])^2 \cdot [\omega]^{n-2}.
\end{equation}

\end{proposition}

\begin{proof} We break the proof into the following steps.

\medskip

\noindent{\it Step 1.}  Equation (\ref{c1}) and  (\eqref{c12}) follow easily from that observation that $\Ricc(\omega)$ is smooth on $X\setminus D$ and it continuously extends to a closed $(1,1)$ form on $X$, which lies in the class of $c_1(X).$ Therefre, 
$$c_1(X, \omega) = c_1(X) - (1-\beta)[D].$$

\noindent {\it Step 2.} We first introduce a few notations.  Let $\omega_0$ be a smooth K\"ahler form in the same class of $[\omega]$. Since the curvature tensor can be viewed as the curvature in the tangent bundle, we write $\theta = (\theta^i_j)$, $\theta_0=(\theta^i_j)$ as the Chern connections on the tangent bundle with respect to the K\"ahler metric $\omega$ and $\omega_0$. Their curvature forms are given by $\Omega$ and $\Omega_0$ with $$\Omega= \sqrt{-1}  \dbar \theta, ~~\Omega_0 = \sqrt{-1} \dbar \theta_0. $$

Let $s$ be a defining section of $D$ and $h$ be a smooth hermitian metric on the line bundle associated to $[D]$. We define
$$X_\epsilon = \{ p\in X~|~ |s|_h^2(p) > \epsilon^2   \}.$$

Then locally $\xi = fS$ for some holomorphic function and on $\partial X_\epsilon$, we have
$$ \xi = fs = |fs|  e^{ i \sigma} = \epsilon|f| h^{-1/2}  e^{\sqrt{-1} \sigma}, ~~\sigma \in [0, 2\pi),$$
and
$$d \xi = \sqrt{-1}  \xi d\sigma  + \epsilon e^{ \sqrt{-1}  \sigma} d (|f|h^{-1/2}).$$
Let $\tau = d(|f|h^{-1/2})$. Then $\tau$ is a smooth $1$-form and on $\partial X_{\epsilon}$,
$$d\xi = \sqrt{-1}  \xi d\sigma + \epsilon \tau|_{\partial X_{\epsilon}}$$

\medskip

\noindent{\it Step 3.}   At each point $p\in \partial X_\epsilon$, we can apply a linear transformation to $(z_1, ..., z_{n-1})$ such that $g_{i\bar j} = \delta_{ij} $ at $p$ and by rescaling $\xi$ so that $g_{n\bar n} = |\xi|^{-2(1-\beta)}$ near $p$. Let $$H_{i\bar j} = g_{i\bar j}, ~~H^{i\bar j} = (H^{-1})_{i\bar j}, ~~H_{n \bar j} = g_{n \bar j}.$$
then we have
\begin{equation}\label{gxi}
|\xi|^{-2(1-\beta)} g^{i\bar n} = |\xi|^{-2(1-\beta)} (g_{n\bar i} + o(1)) (\det(g_{\alpha\beta}) )^{-1} =g_{n\bar i} + o(1) = H^{i \bar j } H_{n\bar j } + o(1)  \end{equation}
and the connection form $\theta$ has following estimates
\begin{eqnarray*}
\theta^n_n &=& \theta^\xi_\xi= \sum g^{n\bar \beta} g_{n\bar \beta, \alpha} d z_\alpha = -(1-\beta +o(1)) \xi^{-1}d\xi + \sum_i o(1)\cdot dz_i \\
&=& -(1-\be) \sqrt{-1} d\sigma +\sum_i O(1)\cdot d z_i \text{ (since }\xi\sim \ep \text{ on } \Dd X_\ep\text {)}\\
\end{eqnarray*}
\begin{eqnarray*}
\theta^i_n &=&   \sum g^{i \bar \beta} g_{n \bar \beta, \alpha} d z_\alpha =  g^{i \bar n} g_{n \bar n , n} d \xi + o(1)\xi^{-1} d\xi + \sum_i o(1)\cdot d z_i  \\
&=&g^{i\bar n}   (|\xi|^{-2(1-\beta)})_\xi d\xi + o(1)\cdot\xi^{-1} d\xi + \sum_i o(1)\cdot d z_i  \\
&=& -(1-\beta) H^{i\bar j} H_{n \bar j} \xi^{-1} d\xi + o(1)\cdot\xi^{-1} d\xi + \sum_i o(1)\cdot d z_i  \\
&=& -\sqrt{-1}(1-\beta) H^{i\bar j} H_{n \bar j} d\sigma + o(1)\cdot d\sigma + \sum_i O(1)\cdot dz_i
\end{eqnarray*}
\begin{eqnarray*}
\theta^n_i &=& O(1)\cdot d \xi +  \sum_i o(1)\cdot dz_i = o(1)\cdot d\sigma +\sum_i o(1)\cdot dz_i \\
\theta^i_k &=& \sum_\alpha O(1)\cdot d z_\alpha =  o(1)\cdot d\sigma + \sum_i O(1)\cdot dz_i.
\end{eqnarray*}
On $\partial X_\epsilon$, by our assumption that $\be\in (1/2,1)$ and using \eqref{gxi} we deduce
\begin{eqnarray*}
\Omega^n_n&=& \sqrt{-1} \sum_{\alpha, \beta} R^n_{n, \alpha \bar \beta} dz_\alpha \wedge d\bar z_\beta \\
&=& \sum_\alpha o(1)\cdot d\sigma \wedge dz_\alpha + \sum_\beta o(1)\cdot d\sigma \wedge d\bar z_\beta + \sum_{i, j} O(1)\cdot dz_i \wedge d\bar z_j
\end{eqnarray*}
\begin{eqnarray*}
\Omega^i_n &=& \left( g^{i\bar \gamma} g_{n\bar \gamma, \alpha} \right)_{\bar \beta} dz_\alpha \wedge d\bar z_\beta\\
&=&  \sqrt{-1} (g^{i\bar n} g_{n\bar n, n})_{\bar z_k} dz_k\wedge d\xi +  \sqrt{-1}(g^{i\bar j } g_{n\bar j, n})_{\bar z_k} dz_k \wedge d\xi +  \sqrt{-1} (g^{i\bar j} g_{n\bar j, k} )_{\bar\xi} d\xi \wedge d z_k \\
&&+ \sqrt{-1}  (g^{i\bar j} g_{n\bar j, k} )_{\bar z_l} d\bar z_l\wedge d z_k  +   \sum_{k,l}o(1)\cdot dz_k \wedge d\bar z_l \\
&=& \sqrt{-1} (1-\be)(H^{i\bar j} H_{n\bar j})_{\bar z_k} dz_k \wedge \xi^{-1} d\xi +   \sum_k o(1)\cdot dz_k \wedge d\sigma + \sum_{k,l}O(1)\cdot dz_k \wedge d\bar z_l\\
&=& \sqrt{-1}(1-\be) \dbar (H^{i\bar j} H_{n\bar j}) \wedge d\sigma +  \sum_k o(1)\cdot dz_k \wedge d\sigma +    \sum_{k,l}O(1)\cdot dz_k \wedge d\bar z_l
\end{eqnarray*}
\begin{eqnarray*}
\Omega^n_i &=&   \sqrt{-1}  R^n_{i, \alpha \bar\beta} dz_\alpha \wedge d\bar z_\beta \\
&=& \sqrt{-1}  R^n_{i, n\bar l} d\xi \wedge d\bar z_l +  \sqrt{-1}  R^n_{i, k \bar n} dz_k \wedge d\bar \xi \\
&& +  \sqrt{-1} R^n_{i, k\bar l} dz_k \wedge d\bar z_l + o(1) dz_k \wedge d\bar z_l\\
&=& o(1) dz_k \wedge d\sigma + o(1) dz_k \wedge d\bar z_l\\
\end{eqnarray*}
\begin{eqnarray*}
\Omega^i_p &=&   \sqrt{-1}  R^i_{p, \alpha \bar\beta} dz_\alpha \wedge d\bar z_\beta \\
&=&  \sqrt{-1} R^i_{p, n\bar l} d\xi \wedge d\bar z_l +  \sqrt{-1}  R^i_{p, k \bar n} dz_k \wedge d\bar \xi \\
&& +  \sqrt{-1}  R^i_{p, k\bar l} dz_k \wedge d\bar z_l + o(1) dz_k \wedge d\bar z_l\\
&=& o(1) dz_k \wedge d\sigma + O(1) dz_k \wedge d\bar z_l
\end{eqnarray*}

\medskip

\noindent {\it Step 4.}
By our assumption, $\omega$ is a smooth K\"{a}hler metric with cone angle $2\pi\be$ along $D\in X$ such that it Ricci curvature is bounded on $X\setminus D$. This implies that
\begin{equation}\label{c1-loc}
\Omega^\al_\al=\Ricc(\omega)=\omega_0-(1-\be) \ddbar \log|s_D|_h^2
\end{equation}
where $h$ is a smooth Hermitian metric on the line bundle $\mathcal{O}_X(D)$, $\omega_0$ is a smooth K\"{a}her metric and $s_D\in H^0(X,\mathcal{O}_X(D))$ is a defining section for $D$. This implies that
$$\int_{X\setminus X_\ep}\Omega^\al_\al\wedge \omega^{n-1} =([\omega_0]-(1-\be)[D])\cdot \omega^{n-1}+o(\ep)\ ,$$
from which we obtain \eqref{c1}.

Let $\theta_t = t\theta + (1-t) \theta_0$ be the connection on the tangent bundle $TX$. The curvature $$\Omega_t =  \sqrt{-1} \bar \partial \theta_t = t \Omega + (1- t) \Omega_0. $$
The transgression formula gives
$$c_2(X, \omega) - c_2(X, \omega_0) = 2  \sqrt{-1} \int_0^1 \dbar P_2(\theta- \theta_0, \Omega_t) dt$$
and
\begin{eqnarray*}
&&  2 P_2(\theta-\theta_0, \Omega_t)\\
&=& (\theta- \theta_0)^n_n \wedge (\Omega_t)^i_i - (\theta-\theta_0)^i_n \wedge (\Omega_t)^n_i +(\theta- \theta_0)^i_i \wedge (\Omega_t)^n_n   -(\theta-\theta_0)^n_i \wedge (\Omega_t)^i_n \\
&&- (\theta-\theta_0)^i_k \wedge (\Omega_t)^k_i\\
%
%&& + o(1) d\sigma \wedge dz_k \wedge d\bar z_l + O(1) dz_i \wedge dz_k \wedge d\bar z_l\\
%
&=& -\sqrt{-1}(1-\be)  (t \dbar \theta  + (1-t) \dbar \theta_0)^i_i \wedge   d\sigma \\
&& - (1-t) (1-\beta)\sqrt{-1} \left( H^{i\bar j} H_{n\bar j} \dbar(\theta_0)^n_i - (\theta-\theta_0)^n_i \wedge \dbar (H^{i\bar j} H_{n\bar j}) \right) \wedge d\sigma\\
&& + o(1) d\sigma \wedge dz_k \wedge d\bar z_l + O(1) dz_i \wedge dz_k \wedge d\bar z_l\\
&=& \sqrt{-1}  (1-\be)\left(-t\dbar ( \theta^i_i)  \wedge d\sigma - (1-t) \dbar \left( (\theta_0)^i_i + H^{i\bar j}H_{n\bar j} (\theta_0)^n_i \right) \wedge d\sigma\right) \\
&& + o(1) d\sigma \wedge dz_k \wedge d\bar z_l + O(1) dz_i \wedge dz_k \wedge d\bar z_l.
\end{eqnarray*}

Now let $\eta$ be a smooth K\"ahler form. Then
\begin{eqnarray*}
&&\int_X (c_2(X, \omega) - c_2(X, \omega_0) \wedge\eta^{n-2}\\
&=&2  \sqrt{-1} \int_0^1 \left(  \int_X \dbar P_2(\theta - \theta_0,  t\Omega + (1-t)  \Omega_0) \wedge \eta^{n-2} \right) dt\\
&=&2    \int_0^1  \left( \int_{\partial X_\epsilon} P_2 (\theta- \theta_0,  t\Omega + (1-t) \Omega_0) \wedge \eta^{n-2} \right)  dt  + o(1)\\
&=&(1-\be)  \sqrt{-1}   \int_0^1   \int_{\partial X_\epsilon}   \left(- t \dbar(\theta^i_i) -   (1-t) \dbar ( (\theta_0)^i_i + H^{i\bar j} H_{n\bar j} (\theta_0)^n_i)  \right)\wedge d\sigma \wedge \eta^{n-2} dt  + o(1)\\
&=&(1-\be)   \sqrt{-1}\int_0^1 \int_D \left( -t \dbar \theta_D - (1-t) \dbar \theta_{0, D} \right) \wedge (\eta|_D) ^{n-2}dt\\
%&&+\int_0^1 \int_D \dbar( (\theta_0)^n_i \cdot H^{i\bar j} H_{n\bar j} )\wedge (\eta|_D) ^{n-2}dt+ o(1)\\
%
&=&(1-\be) (-c_1(D)/2-c_1(D)/2)\cdot [D] \cdot [\eta]^{n-2}   + o(1)   \\
&=&(1-\be) (-c_1(X) + [D])\cdot [D]\cdot [\eta]^{n-2}  + o(1)  .
\end{eqnarray*}
The last three equalities follow from Lemma \ref{c-th}, Lemma \ref{c-th0} and the adjunction formula. 
\medskip
 And similarly, we have
\begin{eqnarray*}
&&\int_X (c_1^2(X, \omega) - c_1^2(X, \omega_0) \wedge\eta^{n-2}\\
&=& \int_0^1 \left(  \int_X \dbar Q(\theta - \theta_0,  t\Omega + (1-t)  \Omega_0) \wedge \eta^{n-2} \right) dt\\
%
%&=& \int_0^1  \left( \int_{\partial X_\epsilon} P_2 (\theta- \theta_0,  t\Omega + (1-t) \Omega_0) \wedge \eta^{n-2} \right)  dt  + o(1)\\
%
&=&2(1-\be)\int_0^1   \int_{\partial X_\epsilon}   \left(- t \dbar(\theta^\al_\al) -   (1-t) \dbar ( (\theta_0)^\al_\al)  \right)\wedge d\sigma \wedge \eta^{n-2} dt  + o(1)\\
%
%&=&2(1-\be)\int_0^1 \int_D \left( -t \dbar \theta_D - (1-t) \dbar \theta_{0, D} \right) \wedge (\eta|_D) ^{n-2}dt\\
%&&+\int_0^1 \int_D \dbar( (\theta_0)^n_i \cdot H^{i\bar j} H_{n\bar j} )\wedge (\eta|_D) ^{n-2}dt+ o(1)\\
%
\text{( by \eqref{c1-loc} )}&=&(1-\be) \bigg(-c_1(X)+(1-\be)[D]-c_1(X)\bigg)\cdot [D] \cdot [\eta]^{n-2}   \\
%
%&=&(1-\be) (-c_1(X) + (1-\be)[D])\cdot [D]\cdot [\eta]^{n-2}-(1-\be)c_1(X)\cdot[D]\cdot[\eta]^{n-2}\\
&=&\bigg (-2(1-\be)c_1(X)\cdot [D]+(1-\be)^2 [D]^2\bigg)\cdot[\eta]^{n-2}
\end{eqnarray*}
which is equivalent to \eqref{c12}.

\noindent{Step 5.}
Suppose that  $\omega_0 \in [\omega]$ is a smooth K\"ahler form. We want to show that
\begin{equation}\label{c2-g}
 \int_X c_2(X, \omega)  \wedge \omega^{n-2}  = \int_X c_2(X, \omega) \wedge \omega_0^{n-2}
\end{equation}
and
\begin{equation}\label{c12-g}
 \int_X c_1^2(X, \omega)  \wedge \omega^{n-2}  = \int_X c_1^2(X, \omega) \wedge \omega_0^{n-2} .
 \end{equation}
Since the proofs are parallel to each other, we will only prove the \eqref{c2-g}.
  Let $\varphi $ be defined by $\omega = \omega_0 + \ddbar \varphi$.
$$ \int_X c_2(X, \omega) \wedge ( \omega^{n-2} - \omega_0^{n-2}) = \sum_{i=0}^{n-3} \int_X \ddbar \varphi   \wedge c_2(X, \omega)\wedge \omega^i \wedge \omega_0^{n-3-i} .$$
For any $i= 0, 1, ..., n-3$,
\begin{eqnarray*}
&& \int_X \ddbar \varphi   \wedge c_2(X, \omega)\wedge \omega^i \wedge \omega_0^{n-3-i} \\
&=&\int_{\partial X_\epsilon} d^c \varphi \wedge c_2(X, \omega)\wedge \omega^i \wedge \omega_0^{n-3-i}  + o(1) .\\
\end{eqnarray*}
Note that
\begin{eqnarray*}
d^c \varphi &=& o(1) d\sigma + O(1) d z_k ,  \\
  \omega &=& o(1) d\sigma \wedge d z_k +o(1) d\sigma \wedge d \bar z_k + O(1) dz_k \wedge d \bar z_l
\end{eqnarray*}
and
\begin{eqnarray*}
&& \Omega^i_i \wedge \Omega^j_j
\sim    \Omega^i_j \wedge \Omega^j_i   \sim  \Omega^n_n \wedge \Omega^j_j
\sim  \Omega^n_j  \wedge \Omega^j_n   \\
&=& o(1) d z_k  \wedge d \bar z_ l \wedge dz_p \wedge  d \sigma +  o(1) d z_k  \wedge d \bar z_l \wedge d \bar z_q  \wedge d\sigma + O(1) dz_k \wedge d\bar z_l \wedge d z_p \wedge  d \bar z_q  .
\end{eqnarray*}
Therefore
\begin{eqnarray*}
&& \int_X \ddbar \varphi   \wedge c_2(X, \omega)\wedge \omega^i \wedge \omega_0^{n-3-i} \\
&=&\int_{\partial X_\epsilon} d^c \varphi \wedge c_2(X, \omega)\wedge \omega^i \wedge \omega_0^{n-3-i} + o(1)\\
&=& 0
\end{eqnarray*}
after letting $\epsilon$ tend to $0$.

 \medskip

 \noindent{Step 6.} Finally, combining the above estimates, we obtain \eqref{c12} and the proof of the proposition is completed.

\end{proof}

One can apply similar argument shows that if $\omega$ is a  smooth conical K\"ahler metric $\omega$ with cone angle $2\pi \beta$ along a smooth simple divisor $D$ and if the Ricci curvature of $\omega$ is bounded, then the $n^{th}$-Chern number $c_n(X, \omega)$ is well-defined and does not depend on the choice of $\omega$ .

\medskip

\subsection{The Gauss-Bonnet and signature theorems for K\"ahler surfaces with conical singularities}

We  first introduce the following definition. 
\begin{definition}\label{Eu-Sig}   Let $X$ be a K\"ahler surface and $\Sigma$ be a smooth holomorphic curve on $X$. If $g$ is a smooth conical K\"{a}hler metric with cone angle $2\pi\beta$ along $\Sigma$, we define the corresponding conical Euler number and signature by
$$\chi(X,g)=\int_{X\setminus \Sigma}   \frac{1}{8\pi^2}\left(\frac{S^2}{24}+|W|^2
-\frac{|\stackrel{\circ}{\Ricc}|^2}{2}
\right)dg 
$$ 
and $$\sigma(X,g)=\frac{1}{12\pi^2}\int_{X\setminus \Sigma}  (|W^+|^2-|W^-|^2)dg,$$
where $S$ is the scalar curvature of metric $g$, $W$ is the Weyl tensor for $g$ and $\stackrel{\circ}{\Ricc}$ is the traceless Ricci curvature.
In particular, if $\beta=1$, we recover classical characteristic class.
\end{definition}

The Gauss-Bonnet and the signature theorems are proved in \cite{AL} for  smooth compact Riemannian $4$-folds  with specified  conical metrics with cone angle $2\pi \beta$ along a smooth embedded Riemann surface.  In particular, As an immediate consequence of Proposition \ref{an-char} and Definition  \ref{Eu-Sig} above, we obtain the following formulas  related to the recent result by Atiyah and LeBrun \cite{AL} by removing the assumption $\beta\in(0,1/3)$ in the K\"ahler case.

\begin{proposition}Let $g$ be a smooth conical K\"{a}hler metric with angle $2\pi\beta$ for $\beta\in (0, 1]$ along a holomorphic curve $\Sigma$. If the Ricci curvature of $g$  is bounded, we have
\begin{eqnarray*}
\chi(X,g) &=&  \chi(X)-(1-\be)\chi(D)\\
\sigma(X,g)&=&\sigma(X)-\frac{1}{3}(1-\be^2)[D]^2.
\end{eqnarray*}

\end{proposition}

\begin{proof}
To prove the statement, we apply the identity (cf. \cite{AL} )
$$\frac{1}{8\pi^2}\left(\frac{S^2}{24}+|W|^2
-\frac{|\stackrel{\circ}{\Ricc}|^2}{2}
\right)dg =c_2(X,g)$$
and 
$$\frac{1}{12\pi^2}(|W^+|^2-|W^-|^2)dg=\frac{1}{3}(c_1^2(X,g)-2c_2 (X,g))$$ 
and the statement follows from Proposition \ref{an-char}

\end{proof}
\subsection{The Chern number inequality on Fano manifolds}

In this section, we will prove Theorem \ref{main4}.

\begin{proposition} \label{chern3} Let $X$ be an $n$-dimensional Fano manifold. If $R(X)=1$, then the Miyaoka-Yau type inequality holds
$$ c_2(X) \cdot c_1(X)^{n-2} \geq \frac{n } {2(n+1)}  c_1(X)^n\ .$$

\begin{proof} We fix a smooth simple divisor $D\in |-mK_X|$ for some $m\in \mathbb{Z}^+$. Such a divisor always exists by Bertini's theorem. Then for any $\beta\in (0,1)$, there exists a smooth conical K\"aher-Einstein metric $\omega$  satisfying $\Ricc(\omega) = \beta g\omega+ (1-\beta)m^{-1} [D]$.

By Chern-Weil theory, if we let
$$\stackrel{\circ} R_{i\bar j k\bar l} = R_{i\bar j k\bar l} - \frac{tr(R)}{n(n+1)} (g_{i\bar j } g_{k\bar l} + g_{i\bar l} g_{k\bar j} ),$$
$$\stackrel{\circ}{Ric}= Ric - \frac{ tr(Ric)}{n} g $$
be the traceless curvature and Ricci curvature tensor,  we have
$$
\left(\frac{2(n+1)}{n} c_2(X, \omega) - c_1^2(X, \omega) \right) \cdot [\omega]^{n-2} = \frac{1}{n(n-1)} \int_X \left( \frac{n+1}{n} |\stackrel{\circ} R |^2 - \frac{ n^2-2}{n^2} |\stackrel{\circ}{Ric}|^2 \right) \omega^n.
$$
We then have 
$$\left(\frac{2(n+1)}{n} c_2(X, \omega) - c_1^2(X, \omega) \right) \cdot [\omega]^{n-2} \geq 0. $$
By Proposition \ref{an-char},  we have
$$c_1(X, \omega) =\beta c_1(X) , ~~~ c_2(X, \omega)= c_2(X)+  (1-\beta) (- c_1 (X) + [D]) \cdot [D]. $$
This implies that
$$\left( \frac{2(n+1)}{n}   c_2(X)  - \beta^2 c^2_1(X) \right) \cdot (\beta c_1(X) ) ^{n-2} \geq  0.$$
The theorem then follows by letting $\beta \rightarrow 1$.

\end{proof}

\end{proposition}

We also have the following lemma when $R(X)<1$ with the same argument in the proof of Proposition \ref{chern3}.

\begin{lemma} \label{chern4}   Let $D\in |-K_X| $ be a smooth simple divisor. If    there exists a conical K\"ahler-Einstein metric $g_\beta$ for some $\beta \in (0, R(X))$ satisfying
$$\Ricc(g) =\beta  g + (1-\beta)  [D],$$
then
$$ c_2(X) \cdot c_1(X)^{n-2} \geq \frac{n \beta^2} {2(n+1)}  c_1(X)^n\ .$$

\end{lemma}

Theorem \ref{main4} is proved by combining Proposition \ref{chern3}, Lemma \ref{chern4} and Theorem \ref{main2}.

\subsection{The Chern number inequality on minimal manifolds of general type}

The following proposition is first claimed in \cite{Ts}, although the analytic part in the proof does not seem to be complete.  The first complete proof seems to be given by  Zhang \cite{ZhY} using the Ricci flow. In this section, we apply Proposition \ref{an-char} to complete Tsuji's original approach.

\begin{proposition} Let $X$ be a smooth minimal model of general type. Then
\begin{equation}
\left( \frac{2(n+1)}{n} c_2(X) - c_1^2(X) \right) \cdot (-c_1(X))^{n-2} \geq 0.
\end{equation}
In particular, if the equality holds, the canonical K\"ahler-Einstein metric  is a complex hyperbolic metric on the smooth part of the canonical model of $X$.

\end{proposition}

\begin{proof}

Fix a smooth simple ample divisor $D$ on $X$. Since $[K_X]+\epsilon [D]$ is a K\"ahler class for any $\epsilon>0$, there exists a smooth conical K\"ahler-Einstein metric in $[K_X] + \epsilon [D]$ with conical singularity along $D$ satisfying
$$\Ricc( \omega ) = -\omega + \epsilon [D]. $$
By the same argument in the proof of Proposition \ref{chern3},
$$ ( \frac{2(n+1)}{n} c_2(X, \omega ) - c_1^2(X, \omega ) ) \cdot (-c_1(X, \omega))^{n-2} \geq \int_X |\stackrel{\circ} R(\omega)|^2 \omega^n\geq 0.$$
By standard argument from \cite{EGZ, ZhZ},  $\omega$ converges  as $\epsilon \rightarrow 0$ to the unique K\"ahler-Einstein metric $g_{can}$   on the canonical model $X_{can}$ of $X$ in $C^\infty$ local topology away from the exceptional locus the pluricanonical system. In particular, $\omega_{can}$ is smooth on the smooth part of $X_{can}$.

By letting $\epsilon$ tends to $0$, we have
\begin{equation} \label{chern6}
 \left( \frac{2(n+1)}{n} c_2(X) - c_1^2(X)  \right) \cdot (-c_1(X))^{n-2} \geq \int_X |\stackrel{\circ}R(\omega_{can})|^2 \omega_{can}^n\geq 0,
 \end{equation}
where $X_{can}^\circ$ is the smooth part of $X_{can}$. When the equality holds, $\stackrel{\circ}R (g_{can})$ vanishes on $X_{can}^\circ$ and so $g_{can}$ must a complex hyperbolic metric on $X^\circ_{can}$.

\end{proof}

Then by the estimate (\ref{chern6}), we immediately obtain an $L^2$-bound for the curvature tensor of the canonical K\"ahler-Einstein metric on the regular part of the canonical model associated to  a smooth minimal model of general type.

\medskip

\section{Discussions}

In this section, we make some speculations on the limiting behavior of the conical K\"ahler-Einstein metrics as $\beta$ tends to $R(X)$.

\begin{conjecture} 

Let $X$ be a Fano manifold with $R(X)=1$. Then by Proposition \ref{rmain1}, there exists a smooth simple divisor $D\in |-mK_X|$ for some $m\in \mathbb{Z}^+$ such that for any $\epsilon>0$, there exists a smooth conical K\"ahler-Einstein metric $g_\epsilon$ with $\Ricc(g_\epsilon) = (1-\epsilon)g_\epsilon + \epsilon m^{-1} [D]$ for all $\epsilon \in (0, 1)$. We conjecture that $(X, g_\epsilon)$ converges to a $\mathbb{Q}$-Fano variety $(X_\infty, g_\infty)$ coupled with a canonical K\"ahler-Einstein metric $g_\infty$ in Gromov-Hausdorff topology.

\end{conjecture}

The above conjecture is related to the recent result in \cite{TW}, where the K\"ahler-Ricci flow is combined with the continuity method to produce a limiting Einstein metric space when $R(X)=1$. We also make a more general conjecture when $R(X)\neq 1$.

\begin{conjecture} Let $X$ be a Fano manifold and $D$ be a smooth simple divisor in $|-mK_X|$ for some $m\in \mathbb{Z}^+$. We consider the conical Ricci flow defined by 
\begin{equation}
\ddt g = - \Ricc(g) + \beta g + m^{-1}(1-\beta) [D]
\end{equation}
starting with a smooth conical K\"ahler metric $g_0\in c_1(X)$ with cone angle $2\pi (1-(1-\beta)m^{-1})$ along $D$. Then for some $m\in \mathbb{Z}^+$ and a generic choice of $D$, we have

\begin{enumerate}

\item If $\beta \in (0, R(X) )$, the flow converges  to a smooth conical K\"ahler Einstein metric on $X$ with conical singularity along $D$.

\medskip
\item If $\beta = R(X)$, the flow converges to a singular K\"ahler-Einstein metric on a paired $\mathbb{Q}$-Fano variety $(X_\infty, D_\infty)$ with conical singularities along an effective $\mathbb{Q}$-divisor $D_\infty \in [-K_{X_\infty}]$.

\medskip
\item If $\beta \in (R(X), 1]$,  the flow converges to a singular  K\"ahler-Ricci soliton on a paired $\mathbb{Q}$-Fano variety $(X_\infty , D_\infty)$ with conical singularities along an   effective $\mathbb{Q}$-divisor $D_\infty  \in [ - K_{X_\infty}]$.

\end{enumerate}

\end{conjecture}

\bigskip

\noindent{\bf Acknowledgements} We would like to thank D.H. Phong, Jacob Sturm,  Valentino Tosatti, Ved Datar and Bin Guo for many stimulating discussions. After we finished the first draft of the paper, we were kindly informed by Tosatti that most of the results in Theorem \ref{main1} and \ref{main2} are independently obtained by Chi Li and Song Sun \cite{LS}. We would also like to thank Chi Li and Song Sun for sending us their preprint and for their interesting comments. 

\bigskip
\bigskip

%%%%%%%%%%%%%%%%%%%%%%%%%%%%%%%%

\end{document}